\documentclass[a4paper,11pt,onecolumn]{amsart}
\usepackage{amsfonts}
\usepackage{amscd, amsmath, amssymb, amsthm, thmtools}
\usepackage{latexsym, verbatim, color, todonotes, enumitem}
\usepackage{mathtools, graphicx, sidecap}
\usepackage{hyperref}
\usepackage{graphicx}
\graphicspath{ {./images/} }
\usepackage[export]{adjustbox}
\usepackage{subcaption}
\usepackage{xypic}
\usepackage{microtype}
\usepackage{verbatim}
\usepackage{todonotes}

\newcommand{\N}{\mathbb N}
\newcommand{\Z}{\mathbb Z}
\newcommand{\Q}{\mathbb Q}
\newcommand{\R}{\mathbb R}
\newcommand{\C}{\mathbb C}
\newcommand{\A}{\mathbb A}

\newcommand{\map}{\rightarrow}
\newcommand{\iso}{\cong}
\newcommand{\ten}{\otimes}
\newcommand{\vect}[1]{\boldsymbol{#1}}
\newcommand{\bigzero}{\mbox{\normalfont\Large\bfseries 0}}
\newcommand{\rvline}{\hspace*{-\arraycolsep}\vline\hspace*{-\arraycolsep}}

\newtheorem{introthm}{Theorem}
\theoremstyle{definition}

\theoremstyle{plain}
\newtheorem{definition}{Definition}[section]
\newtheorem{theorem}[definition]{Theorem}
\newtheorem{lemma}[definition]{Lemma}
\newtheorem{corollary}[definition]{Corollary}
\newtheorem{proposition}[definition]{Proposition}

\theoremstyle{definition}
\newtheorem{example}[definition]{Example}

\theoremstyle{definition}
\newtheorem{remark}[definition]{Remark}

\title{Toric geometry of ReLU neural networks}

\author{Yaoying Fu}

\begin{document}

\begin{abstract}
 Given a continuous piecewise linear function $f:\R^{n_0} \to \R$ and a fixed architecture of feedforward ReLU neural networks, the exact function realization problem asks when a network of that architecture realizes $f$. In contrast, expressivity concerns characterizing all functions exactly realizable by a given architecture. To address these questions systematically, we establish a connection between toric geometry and ReLU neural networks, allowing tools from algebraic geometry to analyze the structures. Starting with an unbiased ReLU network with rational weights, we define the ReLU fan, the ReLU toric variety, and the ReLU Cartier divisor associated with the network. This work also relates the tropical and the toric geometry: the former captures combinatorial data of ReLU networks, while the latter recovers the underlying geometric structure. As an application, we obtain a complete characterization of the functions realizable by unbiased shallow ReLU networks with rational weights. The characterization is expressed in terms of intersection numbers of ReLU Cartier divisors with torus-invariant curves and leads to a decomposition of piecewise-linear functions into shallow and non-shallow components.
\end{abstract}

\maketitle


\section{Introduction}\label{sec:intro}
Although the Universal Approximation Theorem (\cite{Pinkus_1999}) guarantees that neural networks with a single hidden layer and suitable nonlinearity—such as the ReLU activation—can approximate any continuous function on a compact domain to arbitrary accuracy, it does so by allowing arbitrarily large width, and it addresses approximation rather than exact realization. In this paper, we take aim at a more structural question: Which functions are exactly realized by a given neural network architecture, and conversely, given a function, which architectures realize it—and how can this be determined? Our approach to these questions is to develop a dictionary between feedforward ReLU neural networks and toric geometry. This provides a geometric framework in which questions of exact function realization, architectural expressivity, and structural constraints for ReLU neural networks admit natural interpretations through divisors, fans, intersection theory, and toric morphisms. 

The central insight underlying this work is that the function computed by a feedforward ReLU neural network with no biases can be interpreted as the support function of a $\mathbb{Q}$-Cartier divisor on a rational polyhedral fan. Previous works have connected ReLU networks to tropical geometry \cite{zhang2018tropical}.  Our perspective places these observations within the broader framework of toric geometry, where tropical geometry appears naturally through the polyhedral and support-function data associated with divisors on toric varieties.

Previous approaches to realizability and expressivity often formulate compatibility conditions directly in terms of gradients, affine pieces, or variation properties of piecewise linear functions, such as the consistent variation conditions appearing in canonical piecewise-linear representations \cite{1705}. In contrast, the toric framework developed in this paper geometrizes these compatibility conditions through divisors and intersection theory. More precisely, changes of gradients across codimension-one walls are encoded by intersection numbers between Cartier divisors and torus-invariant curves. This viewpoint allows realizability constraints to be interpreted through geometric structures such as divisor classes, toric morphisms, and Picard groups. 

In particular, while variation-based and polyhedral approaches detect compatibility conditions for shallow realizability, the toric framework further organizes shallow-realizable functions into a distinguished subspace of divisor classes. This additional structure allows one to formulate realizability questions using toric morphisms, Picard groups, and intersection theory. As an application, we obtain a decomposition of piecewise linear functions into shallow and non-shallow components, a viewpoint that does not arise naturally from the polyhedral compatibility conditions alone.

In this paper, the toric correspondence is established for unbiased feedforward ReLU neural networks with rational weights. These assumptions ensure that the nonlinear loci determine rational polyhedral fans and that the associated piecewise linear functions define $\Q$-Cartier divisors on the corresponding toric varieties. The restriction to rational weights is also natural from the perspective of classical toric geometry. In the rational polyhedral setting, the semigroups associated with rational polyhedral cones are finitely generated, which underlies many standard constructions in toric geometry. In contrast, irrational polyhedral structures may fail to satisfy these finiteness properties, making the standard toric framework substantially more delicate.

More general ReLU neural networks with biases or irrational weights also exhibit analogous geometric behavior; however, affine shifts destroy positive homogeneity, and one naturally obtains families of translated or non-rational polyhedral complexes rather than a single rational fan centered at the origin. We expect the perspective developed in this paper to extend to these broader settings, although additional geometric complications arise and will be investigated in future work.

We start with feedforward ReLU neural networks with no biases and rational parameters, which we denote by $\mathcal{F}_{\vect{0}}^\Q$. Given an architecture $(n_0,n_1,\ldots,n_k;1)$, we associate to each network a polyhedral fan encoding the regions on which all neurons are simultaneously linear. This assignment defines the \emph{toric encoding map}
\begin{align*}
\mathcal{T}:\mathcal{P} & \map \{\text{ReLU fans}\}\\
\theta &\mapsto \Sigma_{f_\theta}.
\end{align*}

Here $\mathcal{P}$ is the parameter space of a fixed unbiased ReLU neural network architecture with rational weights. A parameter $\theta \in \mathcal{P}$ is mapped to the \emph{ReLU fan $\Sigma_{f_\theta}$} (Definition \ref{def:relufan}).

The ReLU fan $\Sigma_{f_\theta}$ is the polyhedral complex obtained from the minimal subdivision of the input space on which the output function and all hidden-layer neurons are simultaneously linear. Following \cite{grigsby2022transversality}, this polyhedral complex is called the canonical polyhedral complex associated with the network.

More concretely, we use the following example as an illustration of the toric encoding map. 
\begin{example}\label{ex:toricencodingmap}
    Fixing the architecture $(2,3,1;1)$ of unbiased ReLU neural networks, let $\theta=(L_1,\vect{0},L_2,0,L_3,0)$, where $L_1=\begin{bmatrix}
    0 & 1 \\
    0 & -1\\
    1 &-1
\end{bmatrix}$, $L_2=\begin{bmatrix}
    1 & -1 & 1
\end{bmatrix}$ and $L_3=1$, and thus $f_\theta: \R^2 \xrightarrow{\varsigma \circ L_1}\R^3 \xrightarrow{\varsigma \circ L_2}\R\xrightarrow{L_3}\R$. Then the output function of the ReLU neural network $f_\theta$ is  $f=\max\{0,x,y\}$. The first hidden-layer neurons induce the piecewise linear functions
\[
f_1^1=\max\{0,y\},\qquad
f_2^1=\max\{0,-y\},\qquad
f_3^1=\max\{0,x-y\},
\]
while the second hidden layer computes
\[
f_1^2=\max\{0,x,y\}.
\]
The image of $\theta$ under the toric encoding map $\mathcal{T}$ is the ReLU fan $\Sigma_{f_\theta}$:
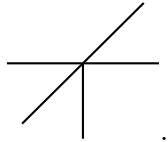
$$
    \begin{tikzpicture}[scale = 0.4]
    \draw[black,thick] (-2.5,0) -- (2.5,0);
    \draw[black, thick] (0,0) -- (0,-2.5);
    \draw[black, thick] (-2,-2) -- (2,2);
    \end{tikzpicture}.
    $$
    $\Sigma_{f_\theta}$ is the canonical polyhedral complex induced by the minimal subdivision of the input space $\R^2$ such that the piecewise linear functions $f,f_1^1,f_2^1,f_3^1$ and $f_1^2$ are linear in each cell of $\Sigma_{f_\theta}$. Thus, the parameter $\theta$ determines a polyhedral fan $\Sigma_{f_\theta}$ on which the output function $f$ is linear on each cone. In later sections, this fan will serve as the underlying fan supporting the ReLU Cartier divisor associated with $f$, providing the bridge from the ReLU network to toric geometry.

\end{example}

After defining the toric encoding map, we define the \emph{ReLU toric variety} (Definition \ref{def:reluvariaty}), which is the toric variety corresponding to the ReLU fan $\Sigma_{f_\theta}$, and the \emph{ReLU Cartier divisor $D_f$} (Definition \ref{def:ReLUdivisor}), which can be understood as the output piecewise linear function of a given feedforward ReLU neural network, up to a shift by a linear term. 

A fundamental feature of the toric framework is that realizability depends only on the associated divisor class, since support functions differing by affine-linear terms determine the same Cartier divisor. Consequently, realizability of piecewise linear functions descends naturally to a divisor-theoretic problem. We denote by $\operatorname{ReLU}^\R(n_0,k)$ to be the set of functions that can be represented by feedforward ReLU neural networks with $k$ hidden layers and real weights. 
\begin{restatable}{introthm}{FunctionDivisor}\label{thm:functiondivisor}
    If $f \in \operatorname{ReLU}^\R(n_0,k)$ and $f'-f=g$, where $g:\R^{n_0} \map \R$ is affine linear, 
    then $f' \in \operatorname{ReLU}^\R(n_0,k)$.
\end{restatable}
Thus, for studying realizability and expressivity at fixed depth, one may equivalently study realizable ReLU Cartier divisors rather than individual piecewise linear representatives.

We next establish a connection between the tropical geometry and toric geometry of ReLU neural networks. It was shown in \cite{zhang2018tropical} that the output piecewise linear function $f$ of a feedforward ReLU neural networks with integral weights is a tropical rational function. We denote by $\operatorname{ReLU}^\Z_{\vect{0}}(n_0,k)$ to be the set of functions that can be represented by feedforward ReLU neural networks with $k$ hidden layers, no biases and integral weights. We prove the following theorem:

\begin{introthm}
    For any $k \in \N$, let $f \in \operatorname{ReLU}^\Z_{\vect{0}}(n_0,k)$ be given and suppose $f_\theta \in \mathcal{F}_{0}^\Q$ realizes $f$. If $f$ is a tropical polynomial, or a convex function in the sense of Definition \ref{def:convexfunction}, then we have
    $$
    P_{-D}=-\text{Newt}(f),
    $$
    and furthermore
    $$
    \text{Vol(Newt}(f))=\text{Vol}(\mathcal{O}_{X_{\Sigma_{f_\theta}}}(-D)).
    $$
\end{introthm}
Here $D$ is the ReLU Cartier divisor, $P_{-D}$ is the polytope associated with the Cartier divisor $-D$, Newt($f$) is the \emph{Newton polytope} associated with the tropical polynomial $f$, Vol(Newt($f$)) is the \emph{normalized volume of the Newton polytope} and Vol$(\mathcal{O}_{X_{\Sigma_{f_\theta}}}(-D))$ is the \emph{volume of the line bundle} $\mathcal{O}_{X_{\Sigma_{f_\theta}}}(-D)$. 

When $f$ is a tropical rational function rather than a tropical polynomial, there is generally no single Newton polytope associated with $f$. In contrast, the divisor-associated polytopes $P_D$ and $P_{-D}$ remain naturally defined in the toric setting. The theorem above shows that, in the tropical-polynomial case, these toric polytopes recover the classical Newton polytope. Thus, divisor-associated polytopes provide a natural extension of Newton polytopes to a broader class of piecewise linear functions arising from ReLU neural networks.

More generally, this theorem illustrates how tropical-geometric constructions fit naturally into the toric framework developed in this paper. While tropical and polyhedral descriptions already capture important aspects of ReLU expressivity, the toric viewpoint supplies additional geometric structures that persist beyond the tropical-polynomial setting. In particular, divisor classes, toric morphisms, and intersection theory provide new tools for studying realizability and expressivity.

As a first application of this perspective, we use intersection theory on ReLU toric varieties to obtain a complete classification of the functions realizable by unbiased shallow ReLU neural networks with rational weights, which we denote by $\operatorname{ReLU}_{\vect{0}}^\Q(n_0,1)$.

\begin{restatable}{introthm}{OneHiddenLayer}\label{thm:onehiddenlayer}
    Let $f \in \operatorname{ReLU}_{\vect{0}}^\Q(n_0,1)$ and suppose $f_\theta \in \mathcal{F}_{0}^\Q$ realizes $f$. The reduced ReLU representation of $f$ is $f_\theta^{\text{red}}: \R^{n_0} \xrightarrow{\varsigma \circ L_1} \R^{n_1} \xrightarrow{L_2}\R$. Let $H^{(1)}_1,\ldots,H^{(1)}_{n_1}$ denote the $n_1$ hyperplanes in $\R^{n_0}$. Then for any two walls, $\tau_1,\tau_2 \subseteq H^{(1)}_i$ for $i\in\{1,\ldots,n_1\}$, we have
    $$
    D_f\cdot V(\tau_1)=D_f \cdot V(\tau_2),
    $$ where $D_f$ is the ReLU Cartier divisor supported on the ReLU fan $\Sigma_{f^\text{red}_\theta}$. 
\end{restatable}

Here, the reduced ReLU representation of a function realizable by an unbiased shallow ReLU neural network refers to the architecture where all weights are integral, and the weight matrix for the hidden layer
\begin{itemize}
    \item has no zero row,
    \item has row-wise greatest common divisor 1, and
    \item no row is parallel to another of the same direction.
\end{itemize} The notation $D_f \cdot V(\tau_1)$ and $D_f \cdot V(\tau_2)$ computes the intersection numbers of the ReLU Cartier divisor $D_f$ with the curve $V(\tau_1)$ and $V(\tau_2)$ respectively.

Geometrically, the intersection numbers measure the change of slope of the support function across codimension-one walls of the fan. Thus, the equalities in Theorem \ref{thm:onehiddenlayer} may be interpreted as global compatibility conditions governing how the piecewise linear function bends across its nonlinear locus.

Theorem \ref{thm:onehiddenlayer} shows that shallow realizability imposes strong intersection-theoretic compatibility constraints on the associated ReLU Cartier divisor. Conversely, we prove that these conditions are also sufficient.
\begin{restatable}{introthm}{OneHiddenLayerConverse}\label{thm:onehiddenlayerconverse}
    Let $f:\R^{n_0} \map \R$ be a finitely-piecewise linear function. If extending the codimension one loci of the non-linear locus of $f$ to be full hyperplanes yields a fan $\Sigma$ in $\R^{n_0}$, and the ReLU Cartier divisor $D_f$, supported on the fan $\Sigma$, satisfies the following:
    $$
    D_f \cdot V(\tau_1)=D_f\cdot V(\tau_2)
    $$
    for any two walls $\tau_1,\tau_2$ coming from the same full hyperplane in $\Sigma$, then $f \in \operatorname{ReLU}_{\vect{0}}^\Q(n_0,1)$.
\end{restatable}

Theorems \ref{thm:onehiddenlayer} and \ref{thm:onehiddenlayerconverse} admit a further geometric reinterpretation through toric morphisms. In particular, shallow realizability corresponds to divisor classes generated by pullbacks of $\mathcal{O}_{\mathbb{P}^1}(1)$ under toric morphisms to products of projective lines. Consequently, the shallow-realizable functions form a distinguished subspace inside the corresponding Picard group of the ReLU toric variety.

This perspective allows the toric framework to go beyond compatibility criteria alone. In particular, piecewise linear functions admit divisor-theoretic decompositions into shallow and non-shallow components through the geometry of the associated toric variety. Thus, the toric viewpoint not only detects shallow realizability conditions, but also organizes realizability geometrically through divisor classes, toric morphisms, and Picard-group structures.

Necessary background on toric geometry is reviewed in \S \ref{sec:toric}. Readers familiar with toric geometry may safely skim portions of these preliminaries and refer back as needed. For readers from other areas, Examples \ref{ex:affinealgebrarelation}, \ref{ex:a1}, \ref{ex:fanofp1} and \ref{ex:relutoric} provide concrete illustrations of the main geometric constructions used throughout the paper. Additional introductory examples from toric geometry are collected in the appendix.

While the present work focuses on the underlying geometric theory, an interesting direction for future work is the algorithmic computation of ReLU fans and associated divisor-theoretic invariants, together with implementations based on existing computational software.

The paper is organized as follows. In \S \ref{sec:relatedwork}, we discuss related works on expressive power of ReLU neural networks and tropical geometry of deep neural networks. Background on ReLU neural networks and toric geometry is introduced in \S \ref{sec:ReLU} and \S \ref{sec:toric}. The toric geometry framework for ReLU neural networks is developed in \S \ref{sec:ReLUtoric}, while \S \ref{sec:tropicaltoric} establishes its relationship with tropical geometry. Finally, in \S \ref{sec:onehiddenlayer}, we apply the toric framework to obtain a complete characterization of functions realizable by unbiased shallow ReLU neural networks through intersection-theoretic conditions on ReLU Cartier divisors. We further reinterpret shallow realizability through toric morphisms and Picard-group structures, leading to divisor-theoretic decompositions of piecewise linear functions into shallow and non-shallow components.

\subsection*{Acknowledgements} The author was partially supported by NSF grants DMS 2133822 and 2001089. The author would like to express her sincere gratitude towards her advisors Qile Chen and Kathryn Lindsey for their cherishable advice and tremendous support. The author would also like to thank Karina Cho for studying \emph{Tropical Geometry} together; Tingting Fang and Zijian Han for discussion of proof of Theorem ~\ref{thm:onehiddenlayer}; Yufan Ren for great inspiration from computer science point of view; Shelby Cox, Bella Finkel, Elisenda Grigsby, Jiayi Li, Jose Israel Rodriguez, Rishi Sonthalia and Maximilian Wiesmann for invaluable discussions and encouragement.   

\section{Related Work}\label{sec:relatedwork}

The exact realization and expressivity of ReLU neural networks have been studied from several complementary perspectives, including max-function representations, polyhedral geometry and tropical geometry. Since ReLU networks compute continuous piecewise linear functions, a central question is how architectural constraints such as depth, width, and the presence or absence of biases restrict the class of piecewise linear functions that can be represented exactly.

One line of work studies exact realization through decompositions of continuous piecewise linear functions into maxima of affine functions. In \cite{wang2005generalization}, it was shown that every continuous piecewise linear function on $\R^n$ can be expressed using linear combinations of $(n+1)$-term max functions. Building on this perspective, \cite{arora2018understanding} proved that every continuous piecewise linear function on $\R^n$ can be represented by a ReLU neural network with depth $\lceil \log_2(n+1)\rceil$, and conjectured that this depth bound is optimal. Later, \cite{hertrich2023towards} related this optimality problem to the non-realizability of certain max functions by networks of smaller depth. For integral weights, \cite{haase2023lower} proved the optimality of the bound $\lceil \log_2(n+1)\rceil$. For real weights, however, \cite{bakaev2025better} showed that this bound is not optimal by proving a stronger lower bound. These works show that depth lower bounds are closely connected to the polyhedral structure of piecewise linear functions.

Another closely related viewpoint studies piecewise linear functions through gradients and variation conditions. Canonical piecewise-linear representations and the associated consistent variation conditions were developed in \cite{1705}. Such conditions encode compatibility constraints among affine pieces of a piecewise linear function. In this paper, we reinterpret related compatibility phenomena through toric geometry: changes of gradients across codimension-one walls are encoded by intersection numbers of Cartier divisors with torus-invariant curves. Thus, while gradient-based and variation-based approaches describe realizability constraints directly on the piecewise linear function, the toric framework packages these constraints into divisor-theoretic and intersection-theoretic data.

Tropical geometry provides another natural language for studying ReLU neural networks. The connection between tropical geometry and ReLU networks was established in \cite{zhang2018tropical}, where functions computed by ReLU networks with integer weights were interpreted as tropical rational functions. This perspective allows tools from polyhedral and polytopal geometry to be applied to expressivity questions. In particular, \cite{hertrich2023towards} associates ReLU networks with pairs of Newton polytopes and relates depth constraints to polyhedral properties of these polytopes. Tropical and polyhedral methods also play an important role in depth lower bound results such as \cite{haase2023lower}. More recently, \cite{blm24} studied the real tropical geometry of neural networks through binary classifiers defined by tropical rational functions. They investigated subdivisions of parameter spaces based on the combinatorial structure of decision boundaries and introduced a classification fan encoding the partitions of a dataset, which can be described as the normal fan of an activation polytope. While their work focuses on the geometry of classification and parameter spaces, our approach studies the toric geometry associated with the piecewise linear functions realized by ReLU networks, particularly through Cartier divisors and intersection theory. More broadly, the number and arrangement of linear regions of piecewise linear neural networks have been studied in works such as \cite{montufar2014number}, where depth and width are related to the combinatorial complexity of the induced polyhedral decomposition. The present work places these tropical and polyhedral structures into a toric framework by interpreting the output piecewise linear function as the support function of a Cartier divisor, thereby making available additional algebro-geometric structures such as divisor classes, toric morphisms, and intersection theory.

Algebraic geometry has also been used to study neural network models, especially polynomial and algebraic variants of neural networks. For example, works such as \cite{finkel2025activationdegreethresholdsexpressiveness}, \cite{kubjas2024geometry}, and \cite{marchetti2025algebra} use algebraic-geometric methods to study expressivity, parameter spaces, and algebraic structures associated with polynomial neural networks. The present paper differs from these works in that ReLU networks are piecewise linear rather than polynomial. Toric geometry is particularly well suited to this setting because it connects rational polyhedral fans, support functions, Cartier divisors, and intersection theory. In this way, the toric viewpoint provides an algebro-geometric framework adapted to the polyhedral nature of ReLU neural networks.

The perspective developed here is complementary to the existing literature. While tropical and polyhedral viewpoints encode ReLU networks through piecewise linear functions and their associated polytopes, the toric framework places these objects into the richer setting of divisors, morphisms, and intersection theory on toric varieties. This additional structure leads to new formulations of realizability questions. In particular, the characterization of shallow networks obtained later in the paper admits a toric reinterpretation in terms of divisor classes and toric morphisms. Moreover, this viewpoint naturally identifies the shallow-realizable functions as a distinguished subspace of divisor classes, which yields a decomposition of piecewise linear functions into shallow and non-shallow components. 

\section{Background on ReLU Neural Networks}\label{sec:ReLU}
In this section, we review the geometric structures associated with feedforward ReLU neural networks that will later be translated into toric geometry. The key objects are hyperplane arrangements, bent hyperplanes, and canonical polyhedral complexes. Together, these encode the subdivision of the input space into regions on which the network computes affine-linear functions. These polyhedral structures will serve as the bridge between ReLU neural networks and the toric-geometric constructions introduced in later sections.

\begin{definition} \label{def:ReLUoperator}
    For any $n \in \N$, the \textit{rectified linear unit} (ReLU) is the map $\operatorname{ReLU}:\R \map \R$ defined by $\operatorname{ReLU}(x)=\max\{0,x\}$. Let $\varsigma: \R^n \to \R^n$ denote the function that applies $\operatorname{ReLU}$ to each coordinate:
    $$
    \varsigma(\vect{x})=(max\{0,x_1\},max\{0,x_2\},\ldots,max\{0,x_n\})
    $$
    where $\vect{x}=(x_1,x_2,\ldots,x_n)$. 
\end{definition} 
\begin{definition} \label{def:ReLU}
    For any number of hidden layers $k\in \N$, a $(k+1)$-layer \emph{feedforward ReLU neural network} is defined as follows:
    $$
    f_\theta:\R^{n_0} \xrightarrow{\varsigma \circ A_1}\R^{n_1} \xrightarrow{\varsigma \circ A_2} \R^{n_2} \xrightarrow{\varsigma \circ A_3} \cdots \xrightarrow{\varsigma \circ A_k} \R^{n_k} \xrightarrow{A_{k+1}} \R
    $$
    where $n_0,n_1,\ldots,n_k \in \N$, and $A_i:\R^{n_{i-1}} \map \R^{n_i}$ are all affine-linear maps. Such a neural network $f_\theta$ is said to be of \emph{architecture} $(n_0,n_1,\ldots,n_k;1)$, \emph{depth} $k+1$, \emph{size} $\sum_{j=1}^k n_j$, $n_i$ is the \emph{width} of the $i$th layer, and each $\R^{n_i},\forall 1\leq i\leq k$, is called a \emph{hidden layer}. Moreover, each affine map $A_i$ is of the form
    $$
    A_i = \begin{bmatrix} M_i \rvert \vect{b_i} \end{bmatrix}
    $$
    where $M_i \in$ Mat$_{{n_i},{n_{i-1}}}(\R)$ and entries are called \emph{weights}, while $\vect{b_i} \in \R^{n_i}$ is called the \emph{bias} of the $i$th layer. A feedforward ReLU neural network is called \emph{shallow} if $k=1$ and \emph{deep} otherwise. Here the ReLU map $\varsigma$ is called the \emph{activation function}.
\end{definition}
\begin{definition}\label{def:ReLUrepresentation}
    We denote by $\mathcal{F}_0^{\mathbb{Q}} = \{f_\theta\}_{\theta}$ the set of all feedforward ReLU neural network representations of functions with no biases and rational parameters -- that is, the set of all formal compositions 
 \begin{equation} \label{eq:formalcomposition}  f_\theta:\R^{n_0} \xrightarrow{\varsigma \circ A_1}\R^{n_1} \xrightarrow{\varsigma \circ A_2} \R^{n_2} \xrightarrow{\varsigma \circ A_3} \cdots \xrightarrow{\varsigma \circ A_k} \R^{n_k} \xrightarrow{A_{k+1}} \R \end{equation}
  where $n_0,\ldots,n_k$ are natural numbers, each $A_i$ is a matrix with entries in $\mathbb{Q}$ that represents a \emph{linear} map, and $\varsigma$ is ReLU map.  We
use the subscript $\theta$ in the the notation $f_\theta$ to encode  
 the information of how the function $f$ is represented as a neural network; concretely, $f$ denotes the underlying function, while
$f_\theta$ records a specific neural-network representation of $f$ in the form \eqref{eq:formalcomposition}.
\end{definition}

The nonlinear behavior of a ReLU neuron occurs precisely when its preactivation changes sign. Consequently, each neuron naturally determines a hyperplane given by the vanishing of its preactivation. Collections of neurons therefore give rise to hyperplane arrangements, which encode where the network changes its affine-linear behavior.
\begin{definition}[ \cite{grigsby2022transversality}]\label{def:hyperplanearrangement} 
    Let 
$$
    f_\theta:\R^{n_0} \xrightarrow{F_1 \coloneqq \varsigma \circ A_1}\R^{n_1} \xrightarrow{F_2 \coloneqq \varsigma \circ A_2} \R^{n_2} \xrightarrow{F_3 \coloneqq \varsigma \circ A_3} \cdots \xrightarrow{F_k \coloneqq \varsigma \circ A_k} \R^{n_k} \xrightarrow{A_{k+1}} \R
$$
be a ReLU neural network. For each layer map $F_i=\varsigma\circ A_i$ and each row
\[
A_{i,j}
=
(M_{i,j}\mid b_{i,j})
\] of $A_i$, consider the affine-linear function
\[
\ell_j^{(i)}(\vect{x})
=
(M_{i,j}\mid b_{i,j})\cdot(\vect{x}\mid 1).
\] If $\ell_j^{(i)}$ is non-constant, its zero locus
\[
H_j^{(i)}
=
\{\vect{x}\in\R^{n_{i-1}}
\mid
\ell_j^{(i)}(\vect{x})=0\}
\]
is called a hyperplane associated with the $i$th layer. The collection of all such hyperplanes is the hyperplane arrangement associated with the $i$th layer
\[
\mathcal A^{(i)}\coloneqq \{H_1^{(i)},\ldots,H_{m_i}^{(i)}\}.
\]
\end{definition}

\begin{definition}\label{def:parameterspace}
    For an architecture $(n_0,n_1,\ldots,n_k;1)$, we define the \emph{parameter space}
    $$
    \mathcal{P}\coloneqq \R^D
    $$
    where a parameter $\theta \coloneqq (M_1,b_1,\ldots,M_k,b_k,A_{k+1})$ consists of all weights in the weight matrices $M_i$, biases $b_i,\forall i \in\{1,\ldots,k\}$ and entries of the affine map of the last layer $A_{k+1}$. Accordingly, $D=n_k+1+\sum_{i=1}^k n_i(n_{i-1}+1)$.
\end{definition}

Every feedforward ReLU neural network computes
a continuous finitely piecewise linear function, meaning that there exists a finite polyhedral subdivision of the input space on which the output is affine-linear.

Convex piecewise linear functions play a distinguished role throughout the paper. They arise naturally as support functions of polytopes and later appear as support functions of torus-invariant divisors.
\begin{definition}\label{def:convexfunction}
    A function $f:\R^n \map \R$ is \emph{convex} if for all $x,y\in \R^n$, and all $\lambda \in [0,1]$, we have 
        $$
        f(\lambda x+(1-\lambda)y) \leq \lambda f(x)+(1-\lambda)f(y)
        $$

    We say that $f$ is \emph{concave} if $-f$ is convex.
\end{definition}
The notion of convexity used here is the standard notion from convex analysis. Later, in the toric-geometric setting, we will also encounter support functions satisfying a toric notion of convexity, which corresponds to concavity in the classical sense.

For each neuron in the $i$th hidden layer, let
$H_j^{(i)} \subseteq \R^{n_{i-1}}$
denote the corresponding hyperplane.

For $i\ge 2$, the preimage
\[
(F_{i-1}\circ\cdots\circ F_1)^{-1}(H_j^{(i)})
\subseteq \R^{n_0}
\]
is called a \emph{bent hyperplane}. Geometrically, bent hyperplanes record
where nonlinearities from deeper layers appear when viewed in the input
space.

Together the first-layer hyperplanes and all bent hyperplanes induce a polyhedral subdivision of the input space. This subdivision records exactly where the network changes its affine-linear behavior and therefore captures the geometry of the network more faithfully than any individual neuron. The resulting polyhedral complex is called the canonical polyhedral complex.

\begin{definition}[ \cite{grigsby2022transversality}]\label{def:polyhedralcomplex}
   Let $f_\theta$ be a ReLU neural network. The \emph{canonical polyhedral complex} $\mathcal C_{f_\theta}$ is the polyhedral complex of the input space $\R^{n_0}$ induced by all first-layer hyperplanes together with all bent hyperplanes associated with $f_\theta$. The precise inductive construction of $\mathcal C_{f_\theta}$ via level-set complexes is recalled in ~\ref{appendixdef:canonicalpolyhedralcomplex}.
\end{definition}

\begin{example}\label{ex:canonicalpolyhedralcomplex}
    Continuing with Example ~\ref{ex:toricencodingmap},
the canonical polyhedral complex $\mathcal{C}_{f_\theta}$ is as follows:
$$
    \begin{tikzpicture}[scale=0.6]
    \draw[red, thick] (-2.5,0) -- (0,0);
    \draw[black, thin] (-2.5,0) -- (2.5,0);
    \draw[red, thick] (0,0) -- (0,-2.5);
    \draw[black, thin] (-2,-2) -- (2,2);
    \draw[black] (2.1,1.45) node{$H_3^{(1)}$};
    \draw[black] (2.6,-0.45) node{$H_1^{(1)}=H_2^{(1)}$};
    \draw[red] (1.5, -2.3) node{$F_1^{-1}(H_1^{(2)})$};
    \draw[black] (1.5,0.6) node{$\sigma_1$};
    \draw[black] (-0.5,0.8) node{$\sigma_2$};
    \draw[black] (-1.2,-0.6) node{$\sigma_3$};
    \draw[black] (-0.45,-1) node{$\sigma_4$};
    \draw[black] (1.25,-1.25) node{$\sigma_5$};
    \draw[black] (2.4, 0.2) node{$\rho_1$};
    \draw[black] (1.6, 2) node{$\rho_2$};
    \draw[black] (-2.4, 0.2) node{$\rho_3$};
    \draw[black] (-2, -1.55) node{$\rho_4$};
    \draw[black] (-0.3, -2.3) node{$\rho_5$};
    \end{tikzpicture}.
    $$
    The hyperplanes \(H_1^{(1)}=H_2^{(1)}\) and \(H_3^{(1)}\) are shown in black, while the pullback \(F_1^{-1}(H_1^{(2)})\) is shown in red. Together these hyperplanes and bent hyperplanes determine a subdivision of the input space. The associated canonical polyhedral complex $\mathcal C_{f_\theta}$ is obtained by taking all cells of this subdivision together with their face relations.
    
    The resulting maximal cells are the five two-dimensional polyhedra $\sigma_1,\ldots,\sigma_5$. The one-dimensional cells are the rays $\rho_1,\ldots,\rho_5$, and the unique zero-dimensional cell is the origin.

By construction, every hidden-layer neuron and the output function $f_\theta$ are affine-linear on each maximal cell $\sigma_i$. Consequently, $\mathcal C_{f_\theta}$ is the minimal polyhedral subdivision of the input space $\R^2$ on which every hidden-layer neuron and the output function $f_\theta$ are simultaneously affine-linear.
\end{example}
\begin{remark}\label{rmk:unbiased}
Unless otherwise specified, we will work with unbiased ReLU neural networks with rational weights. Under these assumptions, the canonical
polyhedral complex is homogeneous and rational, and therefore naturally carries the structure of a rational polyhedral fan (or generalized fan).

This observation provides the bridge from ReLU neural networks to toric geometry. In \S\ref{sec:ReLUtoric}, we formalize this fan as the associated ReLU fan and use it to construct the corresponding ReLU Cartier divisor.

More general ReLU neural networks with biases or irrational weights admit related geometric structures, but additional technical issues
arise and will not be considered here.
\end{remark}

\section{Background on toric geometry}\label{sec:toric}
In this section, we review the toric geometry needed for the toric interpretation of ReLU neural networks developed later in the paper. Standard references for toric geometry include \cite{cox2011toric}. Our goal is to reinterpret piecewise linear functions computed by ReLU neural networks as support functions of divisors on toric varieties. To accomplish this, we introduce cones and fans, which encode polyhedral subdivisions; affine and global toric varieties, which provide the ambient geometric spaces; Cartier divisors and support functions, which encode piecewise linear functions; and intersection numbers, which measure how these functions bend across codimension-one walls. We focus on the aspects of toric geometry most directly relevant to our applications, while additional background is collected in the appendix ~\ref{appendix:toricbackground}. Readers familiar with toric geometry may safely skim portions of this section and refer back to it as needed.

\subsection{Characters and co-characters}\label{subsec:character}

A toric variety is built from a pair of dual lattices. We therefore begin by introducing the character lattice $M$, the co-character lattice $N$, and their natural bilinear pairing, which underlie the toric-geometric constructions appearing throughout this section. For readers interested in the algebro-geometric origins of these constructions, additional background is provided in Appendix~\ref{appendix:toricbackground}. In the setting of this paper, we use the following concrete description.

Fix an integer $n \geq 1$.  The \emph{cocharacter lattice} is  
$$
N \coloneqq \bigoplus_{i=1}^n \mathbb{Z} e_i,
$$
the free $\mathbb{Z}$-module of rank $n$ with standard basis $e_1,\dots,e_n$. The \emph{character lattice} is the dual free $\mathbb{Z}$-module
$$
M \coloneqq \bigoplus_{i=1}^n \mathbb{Z} e_i^\vee
$$
where the dual basis $e_1^\vee,\ldots,e_n^\vee$ satisfies
$$
\langle e_i, e_j^\vee \rangle = \delta_{ij}.
$$
The pairing $\langle\ ,\ \rangle : N \times M \to \mathbb{Z}$ is given by
$$
\left\langle \sum_{i=1}^n a_i e_i, \ \sum_{j=1}^n b_j e_j^\vee \right\rangle
= \sum_{i=1}^n a_i b_i.
$$
Extending scalars to $\mathbb{R}$ gives the real vector spaces
$$
N_{\mathbb{R}} \coloneqq N \otimes_{\mathbb{Z}} \mathbb{R} \ \cong \ \bigoplus_{i=1}^n \mathbb{R} e_i, 
\quad
M_{\mathbb{R}} \coloneqq M \otimes_{\mathbb{Z}} \mathbb{R} \ \cong \ \bigoplus_{i=1}^n \mathbb{R} e_i^\vee,
$$
with the same coordinatewise pairing
$$
\left\langle \sum_{i=1}^n x_i e_i, \ \sum_{j=1}^n y_j e_j^\vee \right\rangle
= \sum_{i=1}^n x_i y_i.
$$  
The \emph{maximal torus} associated with $N$ is
\[
T_N \coloneqq \operatorname{Hom}_{\operatorname{Grp}}(M,\C^*)
\cong (\C^*)^n.
\]
An element $\phi\in T_N$ assigns to each $m\in M$ a nonzero complex number $\phi(m)\in\C^*$ such that
\[
\phi(m+m')=\phi(m)\phi(m')
\]
for all $m,m'\in M$.

For each $m\in M$, the corresponding \emph{character} is the function
\[
\chi^m:T_N\to\C^*,
\qquad
\chi^m(\phi)=\phi(m).
\]

\subsection{Cones}\label{subsec:cones}

The lattices $M$ and $N$ provide the underlying combinatorial framework for toric geometry. The next step is to introduce polyhedral cones in $N_\R\iso \R^n$, which encode the local combinatorial data of toric varieties. Later, under the rational and bias-free assumptions considered in this paper, the ReLU fans associated with neural networks will be built from such cones.
\begin{definition}\label{def:cone}
    A \emph{convex polyhedral cone} in $N_\R$ is a set $\sigma$ of the form
    $$
    \sigma = \text{Cone}(S) \coloneqq \left \{\sum_{u \in S} \lambda_u u \big \vert \lambda_u \geq 0\right \}\subseteq N_\R,
    $$
    where $S \subseteq N_\R$ is finite. We say that $\sigma$ is generated by $S$. Also, Cone$(\emptyset)=\{0\}$. 
\end{definition}
\begin{definition}\label{def:dualcone}
    Given a polyhedral cone $\sigma \subseteq N_\R$, its \emph{dual cone} is defined by
    $$
    \sigma^\vee\coloneqq \left \{m \in M_\R \big \vert \langle m,u \rangle \geq 0, \forall u \in \sigma\right \}.
    $$
    Furthermore, $\sigma^\vee$ is a polyhedral cone in $M_\R$ and $(\sigma^\vee)^\vee=\sigma$.
\end{definition}

\begin{definition}\label{def:face}
    Given $m \not = 0$ in $M_\R$, we define the hyperplane
    $$
    H_m\coloneqq\{u \in N_\R \big \vert \langle m,u \rangle =0\} \subseteq N_\R.
    $$
    A \emph{face of a cone} $\sigma$ is $\tau=H_m \cap \sigma$ for some $m \in \sigma^\vee$, written $\tau \preceq \sigma$. Faces $\tau \not = \sigma$ are called \emph{proper faces}, written $\tau \prec \sigma$. Note that using $m=0$ shows that $\sigma$ is a face of itself, i.e., $\sigma \preceq \sigma$.
\end{definition}

\begin{definition}\label{def:strongconvexity}
    A polyhedral cone $\sigma$ is \emph{strongly convex} if the origin $\{0\}$ is a face of $\sigma$.
\end{definition}

\begin{definition}\label{def:rationalcone}
    A polyhedral cone $\sigma \subseteq N_\R$ is \emph{rational} if $\sigma=$ Cone$(S)$ for some finite set $S \subseteq N$.
\end{definition}

For the applications developed later, the most important cones are strongly convex rational polyhedral cones. Rationality ensures that the associated semigroups are finitely generated, while strong convexity guarantees that the corresponding affine toric varieties have desirable geometric properties. These cones form the basic building blocks of fans and therefore of toric varieties.

\begin{remark}\label{rmk:stronglyconvexrational}
    A strongly convex rational polyhedral cone is generated by the ray generators,  which are the \emph{first lattice points} of the rays, and the ray generators are called \emph{minimal generators}. A ray of a cone is denoted by $\rho$ and the minimal generator of $\rho$ is denoted by $u_\rho$.
\end{remark}

\subsection{Affine Toric Varieties Associated to Cones}\label{subsec:affinetoricvarieties}

A central principle of toric geometry is that combinatorial objects can be converted into algebraic-geometric objects. Starting from a rational polyhedral cone $\sigma \subseteq N_{\mathbb R}$, one constructs an affine semigroup
\[
S_\sigma=\sigma^\vee\cap M,
\]
then a semigroup algebra $\mathbb C[S_\sigma]$, and finally an affine toric variety $U_\sigma$. Thus, the sequence
\[
\sigma
\longmapsto
S_\sigma
\longmapsto
\mathbb C[S_\sigma]
\longmapsto
U_\sigma
\]
provides a bridge from combinatorial data to algebraic geometry. We now describe this construction in more detail.

\begin{definition}\label{def:ssigma}
    Given a rational polyhedral cone $\sigma \subseteq N_\R$, we define the affine semigroup associated with $\sigma$ to be
    $$
    S_\sigma\coloneqq \sigma^\vee \cap M \subseteq M.
    $$
\end{definition}

The finite generation of $S_\sigma$ is guaranteed by the rationality assumption on $\sigma$ through Gordan's Lemma. This finite generation is one of the fundamental reasons that rational polyhedral cones play a distinguished role in toric geometry and allows one to form the semigroup algebra $\C[S_\sigma]$ using finitely many generators.

\begin{definition} \label{def:semigroupalgebra}

    Given an affine semigroup $S$, the \emph{semigroup algebra} $\C[S]$ is the vector space over $\C$ with $S$ as a basis and multiplication induced by the semigroup structure of $S$.
    That is, $\mathbb{C}[S] = \bigoplus_{m \in S} \C \cdot z^m$ and multiplication is given by $z^m \cdot z^{m'} = z^{m+m'}$.
  \end{definition}

The semigroup algebra often admits a concrete presentation. The following example illustrates this construction.

\begin{example}\label{ex:affinealgebrarelation}
    Consider the cone $\sigma=$ Cone$(e_1,e_2,e_1+e_3,e_2+e_3)\subseteq N_\R\iso\R^3$. Then $\sigma^\vee=$ Cone$(e_1^\vee,e_2^\vee,e_3^\vee,e_1^\vee+e_2^\vee-e_3^\vee)\subseteq M_\R \iso \R^3$. Here the affine semigroup $S_\sigma=\N\mathcal{A}$, where $\mathcal{A}=\{e_1^\vee,e_2^\vee,e_3^\vee,e_1^\vee+e_2^\vee-e_3^\vee\}$. Setting $x_1=\chi^{e_1^\vee},x_2=\chi^{e_2^\vee},x_3=\chi^{e_3^\vee},x_4=\chi^{e_1^\vee+e_2^\vee-e_3^\vee}$, we have that
    $$
    \C[S_\sigma]=\C[x_1,x_2,x_3,x_4]/(x_1x_2-x_3x_4).
    $$
\end{example}

The semigroup algebra $\C[S_\sigma]$ provides the algebraic structure associated with the cone $\sigma$. A central idea of algebraic geometry is that one may recover a geometric object from a commutative algebra. In the present setting, this construction produces the affine toric variety associated with $\sigma$.

\begin{definition}\label{def:affinetoricvariety}
     Let $\sigma \subseteq N \ten_\Z \R=N_\R \iso \R^n$ be a rational polyhedral cone with associated affine semigroup $S_{\sigma}=\sigma^\vee \cap M$. Then
    $$
    U_\sigma\coloneqq \text{Hom}_{\text{Mon}} (S_\sigma, \C)
    $$
    is an \emph{affine toric variety}.
\end{definition}

The description
\[
U_\sigma=\operatorname{Hom}_{\mathrm{Mon}}(S_\sigma,\mathbb C)
\]
provides a concrete way to view an affine toric variety as a space of points. Since
\(S_\sigma\) is finitely generated, a monoid morphism is determined by finitely many values, which serve as coordinates on \(U_\sigma\). This perspective will be useful in the example below.
\begin{example}\label{ex:a1}
    $$
    \begin{tikzpicture}
    \draw[black,thick] (0,0) -- (2.5,0);
    \filldraw[black] (0,0) circle (2pt) node[anchor=north]{0};
    \filldraw[black] (1,0) circle (2pt) node[anchor=north]{$e_1$};
    \end{tikzpicture}
    $$
    Let $\sigma=$ Cone$(e_1) \subseteq N_\R \iso \R$. Then $\sigma^\vee=$ Cone$(e_1^\vee)\subseteq M_\R \iso \R$. Setting $\chi^{e_1^\vee}=x$, we have $U_\sigma = \A^1_x$ with coordinate ring $\C[x]$.
\end{example}

\subsection{Fans and their associated toric varieties}\label{subsec:fan}
Fans provide the combinatorial data that encode global toric varieties. Whereas a single cone determines an affine toric variety, a fan specifies how several affine pieces fit together to form a global toric variety. This construction will play a central role later in the paper, where ReLU fans are used to associate toric-geometric objects to neural networks.
\begin{definition} \label{def:fan}
    A \emph{fan} $\Sigma$ in $N_\R$ is a finite collection of cones $\sigma$ such that:
    \begin{enumerate}
        \item Every $\sigma \in \Sigma$ is a strongly convex rational polyhedral cone.
        \item For all $\sigma \in \Sigma$, each face of $\sigma$ is also in $\Sigma$.
        \item For all $\sigma_1,\sigma_2 \in \Sigma$, the intersection $\sigma_1 \cap \sigma_2$ is a face of both(hence also in $\Sigma$).
    \end{enumerate}
    Furthermore, if $\Sigma$ is a fan, then: \begin{itemize}
        \item The \emph{support} of $\Sigma$ is $\lvert \Sigma \rvert=\cup_{\sigma \in \Sigma} \sigma \subseteq N_\R$.
        \item $\Sigma(r)$ is the set of $r-$dimensional cones of $\Sigma$.
        \item A cone $\sigma$ is a \emph{wall} if $\sigma \in \Sigma(n-1)$.
        \item $\Sigma$ is \emph{complete} if 
    $$
    |\Sigma|=\cup_{\sigma \in \Sigma}\sigma=N_\R.
    $$
    \end{itemize}
\end{definition}

Dropping the strong convexity requirement in Definition \ref{def:fan} yields the notion of a generalized fan. The cones arising from ReLU neural networks need not be strongly convex, so generalized fans arise naturally in this setting. For simplicity, however, we restrict our attention to the fan case in the main results, although many of the constructions extend naturally to generalized fans.

Given a fan $\Sigma$, one may associate a toric variety $X_\Sigma$ by gluing the affine toric varieties corresponding to the cones of $\Sigma$. Details of the gluing construction are standard and are recalled in Appendix ~\ref{subsec:toricvarietyfan}.
\begin{definition}\label{def:toricevarietyoffan}
    Let $\Sigma$ be a fan in $N_\R$. Then $X_\Sigma$ is the toric variety associated with the fan $\Sigma$.
\end{definition}

\begin{example}\label{ex:fanofp1} (\textit{Illustration of the construction $X_\Sigma=\mathbb P^1$})
    $$
    \begin{tikzpicture}
    \draw[black,thick] (-2.5,0) -- (2.5,0);
    \filldraw[black] (0,0) circle (2pt) node[anchor=north]{0};
    \filldraw[black] (1,0) circle (2pt) node[anchor=north]{$e_1$};
    \filldraw[black] (-1,0) circle (2pt) node[anchor=north]{$-e_1$};
    \end{tikzpicture}
    $$ 
    Let $\Sigma=\{\{0\}, \sigma_1\coloneqq\text{Cone}(e_1),\sigma_2 \coloneqq\text{Cone}(-e_1)\}$ in $N_\R \iso \R$. As shown in Example \ref{ex:a1}, setting $\chi^{e_1^\vee}=x$ and hence $\chi^{-e_1^\vee}=x^{-1}$, we have $U_{\{0\}}=T_N\iso \C^*$, $U_{\sigma_1}=\A^1_x$ and $U_{\sigma_2}=\A^1_{x^{-1}}$. Note that $\{0\}=\sigma_1 \cap \sigma_2$, we have $S_{\{0\}}=M=\Z=\N\{e_1^\vee\}+\N\{-e_1^\vee\}=S_{\sigma_1}+S_{\sigma_2}$, and $$U_{\{0\}}=\C^*_{x^\pm} = \A^1_x \cap \A^1_{x^{-1}}=U_{\sigma_1} \cap U_{\sigma_2}.$$ Therefore, $U_{\sigma_1}$ and $U_{\sigma_2}$ can be glued along their common open subset $U_{\{0\}}$. The resulting toric variety is $X_\Sigma\cong\mathbb P^1$. Equivalently, one may view this construction as gluing the two affine charts $\A^1_x$ and $\A^1_{x^{-1}}$ along $\C^*$.
\end{example}

\subsection{Cartier divisors and support functions}\label{subsec:cartier}

We have seen how a fan $\Sigma$ determines a toric variety $X_\Sigma$. To study these varieties, we now introduce Cartier divisors and their associated support functions.

The output of a feedforward ReLU neural network is a piecewise linear function. In toric geometry, piecewise linear functions are naturally encoded by Cartier divisors through their support functions. This correspondence forms the foundation of the toric interpretation of ReLU neural networks developed in the subsequent sections.

For toric varieties, Cartier divisors admit a combinatorial description in terms of lattice data assigned to the cones of a fan. Since this combinatorial viewpoint is the one most relevant to our applications, we adopt it throughout the paper. Additional background on prime divisors, Weil divisors, and the divisor-theoretic formulation of Cartier divisors is collected in Appendix~\ref{subsec:additionaldivisors}.

\begin{definition} \label{def:cartierdivisor}
A \emph{Cartier divisor} $D$ on $X_\Sigma$ is specified by a collection
\[
\{m_\sigma\}_{\sigma\in\Sigma(n)}
\]
with $m_\sigma\in M$ satisfying the compatibility condition that
\[
m_\sigma-m_{\sigma'}
\]
vanishes on $\sigma\cap\sigma'$ for all maximal cones
$\sigma,\sigma'\in\Sigma(n)$.
The collection $\{m_\sigma\}$ is called the \emph{Cartier data}.
\end{definition}

Cartier divisors and support functions are two equivalent ways of encoding the same combinatorial data. The Cartier data $\{m_\sigma\}$ naturally determines a piecewise linear function on the support of the fan, called the support function.

\begin{definition} \label{def:supportfunction}
    A \emph{support function} on a fan $\Sigma$ is a function
\[
\varphi:|\Sigma|\to\R
\]
that is linear on each cone of $\Sigma$. A support function is said to be \emph{integral} if
\[
\varphi(|\Sigma|\cap N)\subseteq \Z.
\]
\end{definition}

\begin{proposition}\label{prop:supportfunction}
    Let $D$ be a Cartier divisor with Cartier data
\[
\{m_\sigma\}_{\sigma\in\Sigma(n)}.
\]
Then
\[
\varphi_D(u)=\langle m_\sigma,u\rangle, \text{ when } u\in\sigma,
\]
defines an integral support function on $|\Sigma|$.

Conversely, every integral support function on $\Sigma$ determines a Cartier divisor.
\end{proposition}

\begin{definition}\label{def:qcartier}
A \emph{$\Q$-Cartier divisor} is a divisor $D$ such that $mD$ is Cartier for some positive integer $m$.
\end{definition}

\begin{remark}\label{rmk:qcartierdivisor}
Since we allow rational coefficients in ReLU neural networks, $\Q$-Cartier divisors arise naturally throughout this paper.
\end{remark}

To a Cartier divisor $D$ one associates a line bundle $\mathcal O_{X_\Sigma}(D)$. The support function $\varphi_D$ also determines a polyhedron whose lattice points encode the
global sections of $\mathcal O_{X_\Sigma}(D)$.

\begin{definition}\label{def:polytopeofdivisor}
    (Polyhedron of a divisor) Let $D$ be a Cartier divisor with support function $\varphi_D$.
The associated polyhedron is
\[
P_D
=
\{m\in M_\R \mid
\langle m,u\rangle \ge \varphi_D(u)
\text{ for all } u\in |\Sigma|
\}.
\]
\end{definition}

Since all fans considered in this paper are complete,
the polyhedron $P_D$ is in fact a polytope.

Moreover, the global sections are given by
\[
\Gamma(X_\Sigma,\mathcal O_{X_\Sigma}(D))
=
\bigoplus_{m\in P_D\cap M}
\mathbb C\cdot\chi^m.
\]

\section{Toric geometry of ReLU Neural Networks}\label{sec:ReLUtoric}
Having established the necessary background, we now develop a dictionary between feedforward ReLU networks and toric geometry.
\begin{definition}\label{def:relufan}
    Let $f\in \operatorname{ReLU}_{\vect{0}}^\Q(n_0,k)$ and suppose $f_\theta\in \mathcal{F}_{0}^\Q$ realizes $f$:
$$
f_\theta:\R^{n_0} \xrightarrow{\varsigma \circ L_1}\R^{n_1} \xrightarrow{\varsigma \circ L_2} \R^{n_2} \xrightarrow{\varsigma \circ L_3} \cdots \xrightarrow{\varsigma \circ L_k} \R^{n_k} \xrightarrow{L_{k+1}} \R.
$$
The \emph{ReLU fan} associated with $f_\theta$ is defined to be the canonical polyhedral complex
\[
\Sigma_{f_\theta}:=\mathcal C_{f_\theta}.
\]
\end{definition}

\begin{remark}\label{rmk:relufan}
    In general, $\Sigma_{f_\theta}$ is a generalized fan. Throughout this paper, we restrict attention to networks for which $\Sigma_{f_\theta}$ is a fan. This is the generic situation in the examples and applications that motivate our study.
\end{remark}
\begin{proposition}\label{prop:toricecondingmap}
    With Definition \ref{def:relufan} and Remark \ref{rmk:relufan}, fixing an architecture $(n_0,\ldots,n_k;1)$ of unbiased ReLU neural networks with rational weights, there exists the \emph{toric encoding map}
    $$
    \mathcal{T}: \mathcal{P} \to \left\{\text{ReLU fans} \right\}
    $$ with
    $$
    \theta \mapsto \Sigma_{f_\theta}.
    $$
\end{proposition}

To connect the toric and neural-network viewpoints, we identify the lattices appearing in toric geometry with the ambient input space of the neural network.
\begin{remark}\label{rmk:pairing}
    The real vector space $N_\R$ spanned by the co-character lattice $N$ is identified as $\R^{n_0}$, i.e., $N_\R = \R^{n_0}$. Inherited from the standard coordinates of $\R^{n_0}$, denote by $\{e_1,\ldots,e_{n_0}\}$ we get the standard basis of $N$. Identify the character lattice $M$ as $N$ and denote by $\{e_1^\vee,\ldots,e_{n_0}^\vee\}$ we get the standard basis of $M=N^\vee$ via 
    $\langle e_i,e_j^\vee \rangle=\delta_{ij}$. Here, the pairing $\langle , \rangle:N \times M \map \Z$ is the usual dot product.
\end{remark}

Let $f_\theta \in \mathcal{F}_{0}^\Q$ and the corresponding ReLU fan $\Sigma_{f_\theta}$ be given.

\begin{definition}\label{def:reluvariaty}
     Denote by $X_{\Sigma_{f_\theta}}$ the \emph{ReLU toric variety} is defined to be the toric variety corresponding to the ReLU fan $\Sigma_{f_\theta}$. Since the ReLU fan $\Sigma_{f_\theta}$ is \emph{complete}, the ReLU toric variety $X_{\Sigma_{f_\theta}}$ is \emph{proper}.
\end{definition}    
\begin{definition}\label{def:ReLUdivisor}
    The output finitely-piecewise linear function $f$ serves as the support function of a $\Q-$Cartier divisor $D_f$ supported on $\Sigma_{f_\theta}$. Then we define $D_f$ to be the \emph{ReLU Cartier divisor}.
\end{definition}

\begin{lemma}\label{lem:slopevector}
When the ReLU Cartier divisor associated with the output function $f$ is Cartier with \emph{Cartier data} $\{m_\sigma\}_{\sigma \in \Sigma_{f_\theta}(n_0)}$, then we may interpret the Cartier data as the \emph{slope vector} of the linear term. More specifically, say $f:\R^{n_0} \map \R$, then in each maximal dimensional cone $\sigma \in \Sigma_{f_\theta}$, the function $f$ is linear of the form $a_1x_1+\ldots+a_{n_0}x_{n_0}$. We have the slope vector $$(a_1,\ldots,a_{n_0})=m_\sigma.$$
\end{lemma}
\begin{proof}
    $f$ served as the support function $\varphi_{D_f}:\vert \Sigma_{f_\theta} \vert \map \R$ with 
    $$u \mapsto \varphi_{D_f}(u)=\langle m_\sigma,u \rangle \text{ when }u \in \sigma.
    $$ On the other hand, when $u\in \sigma$, $\varphi_{D_f}(u)=a_1x_1+\ldots+a_{n_0}x_{n_0}$ with $u=x_1e_1+\ldots+x_{n_0}e_{n_0}$. It follows that $m_\sigma=a_1e_1^\vee+\ldots+a_{n_0}e_{n_0}^\vee$. This yields the slope vector $(a_1,\ldots,a_{n_0})=m_\sigma$ as desired.
\end{proof}

\begin{example}\label{ex:relucartierdata}
Continuing with Example~\ref{ex:toricencodingmap}, consider the output function
\[
f(x,y)=\max\{0,x,y\}.
\]
On each maximal cone of the ReLU fan $\Sigma_{f_\theta}$, 
$$
    \begin{tikzpicture}[scale=0.6]
    \draw[black, thin] (-2.5,0) -- (0,0);
    \draw[black, thin] (-2.5,0) -- (2.5,0);
    \draw[black, thin] (0,0) -- (0,-2.5);
    \draw[black, thin] (-2,-2) -- (2,2);
    \draw[black] (1.5,0.6) node{$\sigma_1$};
    \draw[black] (-0.5,0.8) node{$\sigma_2$};
    \draw[black] (-1.2,-0.6) node{$\sigma_3$};
    \draw[black] (-0.45,-1) node{$\sigma_4$};
    \draw[black] (1.25,-1.25) node{$\sigma_5$};
    \end{tikzpicture}
    $$
the function $f$ is linear. Its slope vectors on $\sigma_1$ through $\sigma_5$ are
\[
e_1^\vee, \qquad e_2^\vee, \qquad 0, \qquad 0,\qquad e_1^\vee.
\]
Thus, viewing $f$ as a support function on $\Sigma_{f_\theta}$, these slope vectors form the Cartier data of the associated ReLU Cartier divisor $D_f$.
\end{example}

\begin{remark}\label{rmk:slopevector}
    When the ReLU Cartier divisor associated with the output function $f$ is $\Q-$Cartier, we have the same result as in Lemma \ref{lem:slopevector}.
\end{remark}
To study the set of piecewise linear functions realizable by a fixed ReLU architecture, it suffices to study the corresponding ReLU Cartier divisors.

Indeed, two piecewise linear functions $f$ and $f'$ determine the same $\Q$-Cartier divisor whenever they differ by a linear function. Equivalently, if
\[
f-f'=g,
\]
where $g$ is linear, then $f$ and $f'$ define the same divisor. Consequently, there is a natural correspondence
\begin{equation}\label{eq:functiondivisor}
    \{\text{piecewise linear functions}\}/\{\text{linear functions}\}
    \iso
    \{\Q\text{-Cartier divisors}\}.
\end{equation}

For example, the two piecewise linear functions
\[
f(x,y)=\max\{0,x,y\}
\qquad\text{and}\qquad
f'(x,y)=\max\{x+y,2x+y,2y+x\}
\]
differ by the linear function $x+y$. Hence they determine the same $\Q$-Cartier divisor under the correspondence \eqref{eq:functiondivisor}.

To use this correspondence in the study of network expressivity, we must verify that quotienting by linear functions does not affect the depth of a ReLU realization. The following theorem establishes precisely this fact.

\FunctionDivisor*
\begin{proof}
    Since $g$ is affine linear, let $g=\sum_{i=1}^{n_0} a_ix_i+b_g$ where \[\begin{bmatrix}
        x_1 & x_2 & \cdots &x_{n_0}
    \end{bmatrix}^T \in \R^{n_0}
    \] and $a_i,b_g \in \R$ for $i\in\{1,\ldots,n_0\}$. Say $f_\theta:\R^{n_0} \xrightarrow{\varsigma \circ A_1}\R^{n_1} \xrightarrow{\varsigma \circ A_2} \R^{n_2} \xrightarrow{\varsigma \circ A_3} \cdots \xrightarrow{\varsigma \circ A_k} \R^{n_k} \xrightarrow{A_{k+1}} \R$ realizes $f$. Then let $A_1'$ be an $(n_1+2) \times (n_0+1)$ matrix as follows: 
    $$
    A_1' \coloneqq \begin{bmatrix}
        A_1 \\
        \hline
        \begin{matrix}
            a_1 & a_2 & \cdots & a_{n_0} & b_g\\
        -a_1 & -a_2 & \cdots &-a_{n_0} & -b_g
        \end{matrix}
    \end{bmatrix}.
    $$
    For all $j$ from $2$ to $k$, note that $A_j$ is an $n_j \times (n_{j-1}+1)$ matrix which can be written as $A_j=\begin{bmatrix}
        M_j \vert \vect{b_j}
    \end{bmatrix}$. Then let $A_j'$ be an $(n_j+2) \times (n_{j-1}+3)$ matrix as follows:
$$
A_j' \coloneqq \begin{bmatrix}
  M_j
  & \rvline & \bigzero & \rvline &\vect{b_j}\\
\hline
  \bigzero & \rvline &
  \begin{matrix}
  1 & 0\\
  0 & 1
  \end{matrix} & \rvline &\begin{matrix}
      0 \\
      0
  \end{matrix}
\end{bmatrix},\forall j \in\{2,\ldots,k\}.
$$
As for $A_{k+1}'$, we note that $A_{k+1}=\begin{bmatrix}
    M_{k+1} & \rvline &b_{k+1}
\end{bmatrix}$. Then let $A_{k+1}'$ be a $1 \times (n_k+3)$ matrix as follows:
$$
A_{k+1}' \coloneqq \begin{bmatrix}
  M_{k+1}
  & \rvline & 1 & -1 & b_{k+1}
\end{bmatrix}.
$$
Then $f'$ is realizable by the following $k+1$-layer ReLU neural network $f'_\theta:\R^{n_0} \xrightarrow{\varsigma \circ A_1'}\R^{n_1+2} \xrightarrow{\varsigma \circ A_2'} \R^{n_2+2} \xrightarrow{\varsigma \circ A_3'} \cdots \xrightarrow{\varsigma \circ A_k'} \R^{n_k+2} \xrightarrow{A_{k+1}'} \R$, i.e., $f' \in \operatorname{ReLU}(n_0,k)$ as desired.
\end{proof}

\begin{remark}\label{rmk:functiondivisor}
    For Theorem \ref{thm:functiondivisor}, if we restrict ourselves to $\operatorname{ReLU}^\R_{\vect{0}}(n_0,k)$, i.e., where all affine maps $A_i$ are linear maps $L_i$, then we have the same conclusion. That is, if $f \in \operatorname{ReLU}^\R_{\vect{0}}(n_0,k)$ and $f'-f=g$, where $g:\R^{n_0} \to \R$ is linear, then $f' \in \operatorname{ReLU}^\R_{\vect{0}}(n_0,k)$.
    \end{remark}
With Theorem \ref{thm:functiondivisor} and Remark \ref{rmk:functiondivisor}, we know that investigating the expressivity of architectures of a fixed depth, i.e., studying functions that can be realized by architectures of certain depth $k+1$, is equivalent to studying the ReLU Cartier divisors that can be realized by those architectures. 

Before moving to the next section, we illustrate the concepts defined in \S \ref{sec:ReLUtoric} with a specific example as follows.

\begin{example}\label{ex:relutoric}
    Continuing with Example ~\ref{ex:toricencodingmap} and ~\ref{ex:canonicalpolyhedralcomplex},
the ReLU fan $\Sigma_{f_\theta}$ is as follows:
$$
    \begin{tikzpicture}[scale=0.6]
    \draw[red, thick] (-2.5,0) -- (0,0);
    \draw[black, thin] (-2.5,0) -- (2.5,0);
    \draw[red, thick] (0,0) -- (0,-2.5);
    \draw[black, thin] (-2,-2) -- (2,2);
    \draw[black] (2.1,1.45) node{$H_3^{(1)}$};
    \draw[black] (2.6,-0.45) node{$H_1^{(1)}=H_2^{(1)}$};
    \draw[red] (1.5, -2.3) node{$F_1^{-1}(H_1^{(2)})$};
    \draw[black] (1.5,0.6) node{$\sigma_1$};
    \draw[black] (-0.5,0.8) node{$\sigma_2$};
    \draw[black] (-1.2,-0.6) node{$\sigma_3$};
    \draw[black] (-0.45,-1) node{$\sigma_4$};
    \draw[black] (1.25,-1.25) node{$\sigma_5$};
    \draw[black] (2.4, 0.2) node{$\rho_1$};
    \draw[black] (1.6, 2) node{$\rho_2$};
    \draw[black] (-2.4, 0.2) node{$\rho_3$};
    \draw[black] (-2, -1.55) node{$\rho_4$};
    \draw[black] (-0.3, -2.3) node{$\rho_5$};
    \end{tikzpicture}.
    $$

    The output piecewise linear function is $f=$ max$\{0,x,y\}$ or $f=0\oplus x\oplus y$ in tropical algebra. The ReLU Cartier divisor associated with $f$ is $D_f$ with Cartier data, or slope vector $\{m_{\sigma_1}=m_{\sigma_5}=e_1^\vee, m_{\sigma_2}=e_2^\vee,m_{\sigma_3}=m_{\sigma_4}=0\}$. 
\end{example}

\section{From Tropical to Toric Geometry of ReLU Networks}\label{sec:tropicaltoric}

Tropical geometry provides an alternative language for describing convex piecewise-linear functions. We will use only a small amount of tropical notation and refer the reader to \cite{maclagan2021introduction} for further background.

Tropical addition and multiplication are defined by
\[
x\oplus y=\max\{x,y\},
\qquad
x\odot y=x+y.
\]
Accordingly, a tropical monomial
\[
x_1^{u_1}\cdots x_n^{u_n}
\]
represents the linear function
\[
u_1x_1+\cdots+u_nx_n,
\]
and tropical division corresponds to ordinary subtraction.

Throughout this section, we restrict attention to tropical polynomials arising from unbiased ReLU networks. Such tropical polynomials are of the form
\[
f(x)=x^{u_1}\oplus\cdots\oplus x^{u_r},
\]
where \(u_1,\ldots,u_r\in\mathbb Z^{n_0}\).

Viewed as an ordinary function on \(\mathbb R^{n_0}\), \(f\) becomes
\[
f(x)=\max_i\{\langle u_i,x\rangle\}.
\]
In particular, \(f\) is a convex finitely piecewise-linear function in the sense of Definition~\ref{def:convexfunction}.

The Newton polytope associated with \(f\) is
\[
\operatorname{Newt}(f)
=
\operatorname{Conv}(u_1,\ldots,u_r).
\]

The goal of this section is to compare the Newton polytope of a tropical polynomial with the polytope associated to the corresponding ReLU Cartier divisor. To make this comparison, we recall a basic fact from toric geometry concerning convex support functions and basepoint free divisors.
\begin{definition}\label{def:basepointfree}
    If $\varphi_D$ is convex, i.e., $\varphi_D$ is \emph{convex} if
    $$
    \varphi_D\bigl(tu+(1-t)v\bigr) \geq t\varphi_D(u)+(1-t)\varphi_D(v),\forall u,v \in N_\R, t \in [0,1]
    $$
    then the Cartier divisor $D$ is \emph{basepoint free}. Furthermore, 
    $$
    P_D=\text{ Conv}(m_\sigma \vert \sigma \in \Sigma(n)).
    $$ 
\end{definition}

\begin{remark}\label{rmk:convexity}
    The convention above is opposite to the convention adopted in Definition~\ref{def:convexfunction}. Thus a support function that is convex in the toric-geometric sense is concave in the sense of Definition~\ref{def:convexfunction}.
\end{remark}

The preceding observation explains why a sign change naturally appears when passing from tropical geometry to toric geometry. We can now identify the Newton polytope of a tropical polynomial with the polytope associated to the corresponding ReLU Cartier divisor.

\begin{theorem}\label{thm:newtonandpd}
Let $f\in \operatorname{ReLU}_{\vect{0}}^\Z(n_0,k)$, and suppose $f_\theta \in \mathcal{F}^\Q_{0}$ realizes $f$:
$$
    f_\theta: \R^{n_0} \xrightarrow{\varsigma \circ L_1}\R^{n_1} \xrightarrow{\varsigma \circ L_2} \R^{n_2} \xrightarrow{\varsigma \circ L_3} \cdots \xrightarrow{\varsigma \circ L_k} \R^{n_k} \xrightarrow{L_{k+1}} \R.
$$
Let $X_{\Sigma_{f_\theta}}$ and $D$ be the corresponding ReLU toric variety and ReLU Cartier divisor.

If $f$ is a tropical polynomial, then 
$$
P_{-D}=-\text{Newt}(f),
$$
where $P_{-D}$ is the polytope associated with the Cartier divisor $-D$.
\end{theorem}

\begin{proof}
    If $f$ is a tropical polynomial, then $-f$ is convex in the sense of Definition \ref{def:basepointfree}, and the Cartier divisor corresponding to $-f$ is $-D$. We may write $f=\vect{x}^{u_1}\oplus \vect{x}^{u_2} \oplus \ldots \oplus \vect{x}^{u_n}$, where $\vect{x}=(x_1,\ldots,x_{n_0})$ and $u_i \in \Z^{n_0},\forall i$, so $-f=\min\{\langle \vect{x},-u_1\rangle,\langle \vect{x},-u_2\rangle,\ldots,\langle \vect{x},-u_n\rangle \}$. By convexity of $-f$, we know that $-D$ is basepoint free and
    $$
    P_{-D}=\text{Conv}(m_\sigma \vert \sigma \in \Sigma(n_0)),
    $$
    where $m_\sigma$ denotes the Cartier data
    of $-D$. With Lemma \ref{lem:slopevector}, $m_\sigma$'s are identified as slope vectors, which are precisely the exponents $u_i$'s of the tropical polynomial $f$, up to a minus sign and possibly excluding
redundant exponents. Since redundant exponents
do not affect the convex hull, and
$\text{Newt}(f)=\text{Conv}(u_1,u_2,\ldots,u_n)$,
we may conclude that
$P_{-D}=-\text{Newt}(f)$.
\end{proof}
\begin{remark}\label{rmk:minussign}
    There is a minus sign in Theorem \ref{thm:newtonandpd}. The reason behind this is that for a piecewise linear function $f:\R^{n_0} \map \R$, being convex in the sense of Definition \ref{def:basepointfree} means being concave in the sense of Definition \ref{def:convexfunction}.
\end{remark}

Theorem \ref{thm:newtonandpd} identifies the combinatorial objects arising from tropical and toric geometry. We next compare their associated notions of volume.

On the tropical side, the natural invariant is the normalized volume of the Newton polytope. On the toric side, the corresponding invariant is the volume of the associated line bundle.

For a lattice polytope \(P\subseteq M_{\mathbb R}\) of dimension \(n\), the normalized volume
\(\operatorname{Vol}(P)\) is defined to be \(n!\) times its Euclidean volume.
Equivalently,
\[
\operatorname{Vol}(P)
=
n!\lim_{m\to\infty}
\frac{\#(mP\cap M)}{m^n},
\]
by Ehrhart theory.

The toric counterpart of normalized volume is the volume of a line bundle. This invariant measures the asymptotic growth rate of global sections and will allow us to compare the normalized volume of a Newton polytope with the
volume of the corresponding line bundle.

\begin{definition}\label{def:volumeoflinebundle}
    Let $X_\Sigma$ be the toric variety of a fan $\Sigma$ with dim $X_\Sigma=n$ and a Cartier divisor $D$. Then the \emph{volume} of the line bundle $\mathcal{O}_{X_\Sigma}(D)$ is defined to be the nonnegative real number
    $$
    \text{Vol}(\mathcal{O}_{X_\Sigma}(D))\coloneqq \limsup\limits_{m\rightarrow \infty} \frac{h^0(X_\Sigma,\mathcal{O}_{X_\Sigma}(mD))}{m^n / n!}
    $$
    where $h^0(X_\Sigma,\mathcal{O}_{X_\Sigma}(mD))$ is the dimension of the global sections of the line bundle $\mathcal{O}_{X_\Sigma}(mD)$.
\end{definition}

Combining Theorem \ref{thm:newtonandpd} with the asymptotic interpretation of both normalized volume and line bundle volume yields the following correspondence.

\begin{theorem}\label{thm:polynomialpdvolume}
    
Let $f\in \operatorname{ReLU}_{\vect{0}}^\Z(n_0,k)$, and suppose $f_\theta \in \mathcal{F}_{0}^\Q$ realizes $f$:
$$
    f_\theta: \R^{n_0} \xrightarrow{\varsigma \circ L_1}\R^{n_1} \xrightarrow{\varsigma \circ L_2} \R^{n_2} \xrightarrow{\varsigma \circ L_3} \cdots \xrightarrow{\varsigma \circ L_k} \R^{n_k} \xrightarrow{L_{k+1}} \R.
$$
    When the output piecewise-linear function $f$ is a tropical polynomial, we have
    $$
    \text{Vol(Newt}(f))=\text{Vol}(\mathcal{O}_{X_{\Sigma_{f_\theta}}}(-D)),
    $$
    where Vol(Newt$(f)$) is the normalized volume of the Newton polytope, and Vol$(\mathcal{O}_{X_{\Sigma_{f_\theta}}}(-D))$ is the volume of the line bundle $\mathcal{O}_{X_{\Sigma_{f_\theta}}}(-D)$.
\end{theorem}

\begin{proof} \label{proof:volumeoflinebundle}
     As in Definition \ref{def:polytopeofdivisor}, note that $P_{-D}$ gives the global sections of the line bundle $\mathcal{O}_{X_{\Sigma_{f_\theta}}}(-D)$. This amounts to saying that 
    $$
    h^0(X_{\Sigma_{f_\theta}},\mathcal{O}_{X_{\Sigma_{f_\theta}}}(m(-D)))=\# mP_{-D} \cap M.
    $$
    By Definition \ref{def:volumeoflinebundle}, we see that
    \begin{align*}
        \text{Vol}(\mathcal{O}_{X_{\Sigma_{f_\theta}}}(-D))&=\limsup\limits_{m\rightarrow \infty} \frac{h^0(X_{\Sigma_{f_\theta}},\mathcal{O}_{X_{\Sigma_{f_\theta}}}(m(-D)))}{m^{n_0} / n_0!}\\
        &=n_0! \limsup\limits_{m \rightarrow \infty} \frac{\# mP_{-D} \cap M }{m^{n_0}}
    \end{align*}

    When the limit exists, we know that $\limsup\limits_{m\rightarrow \infty}=\lim\limits_{m \rightarrow \infty}$, which yields the following equality:
    \begin{align*}
      \text{Vol}(\mathcal{O}_{X_{\Sigma_{f_\theta}}}(-D))&=n_0! \lim\limits_{m \rightarrow \infty} \frac{\# mP_{-D} \cap M }{m^{n_0}}\\ &=\text{Vol}(P_{-D})=\text{Vol}(-\text{Newt}(f))[\text{by Theorem \ref{thm:newtonandpd}}]\\
      &=\text{Vol(Newt}(f))
    \end{align*}
    as desired.
\end{proof}
\begin{remark}\label{rmk:polynomialpdvolume}
    In Theorem \ref{thm:polynomialpdvolume}, we assume that dim Newt$(f)=$ dim $P_{-D}=n_0$. If this condition fails, then $f$ is linear along some positive-dimensional subspace of the input space $R^{n_0}$. In that case, we may replace the input space by the quotient of $R^{n_0}$ by this linear subspace, so that Newt$(f)$ becomes full-dimensional.
\end{remark}

\begin{example}\label{ex:tropicalpolynomial}
    We conclude with a concrete illustration of the correspondence established above. Continuing with Example \ref{ex:relutoric}, we note that the output piecewise linear function $f=0\oplus x\oplus y$ is a tropical polynomial. By Theorem $\ref{thm:newtonandpd}$, we know that $-f$ is convex in the sense of Definition \ref{def:basepointfree}, thus $D_{-f}=-D_f$ is basepoint free, and
    $$
    P_{-D_f}=\text{Conv}((0,0),(0,-1),(-1,0))=-\text{Newt}(f).
    $$
    Furthermore, by Theorem \ref{thm:polynomialpdvolume}, we have
    $$
    \text{Vol(Newt}(f))=\text{Vol}(\mathcal{O}_{X_{\Sigma_{f_\theta}}}(-D_f))=1.
    $$
\end{example}
\section{Complete classification of functions that can be realized by unbiased shallow ReLU Neural Networks}\label{sec:onehiddenlayer}
The toric interpretation developed in the previous sections allows us to
study shallow ReLU networks through the geometry of the associated ReLU
Cartier divisors.

Recall that Theorems~\ref{thm:onehiddenlayer} and
\ref{thm:onehiddenlayerconverse} characterize functions realizable by
unbiased shallow ReLU networks. In this section, we reinterpret these
results in terms of intersection theory on toric varieties.

The key observation is that the realizability condition can be detected by
examining how the ReLU Cartier divisor intersects torus-invariant curves.
This viewpoint leads naturally to a toric characterization of shallow
networks and will ultimately allow us to decompose an arbitrary piecewise
linear function into shallow and non-shallow components.

On a complete toric variety, the geometry of a Cartier divisor is often
captured by its intersection numbers with torus-invariant curves. Since our classification will ultimately be expressed in terms of the
ReLU Cartier divisor, we first recall the relevant notion of intersection
number.

Let $\Sigma$ be a complete fan. For every wall $\tau=\sigma\cap\sigma'
\in\Sigma(n-1)$, there is an associated torus-invariant curve $V(\tau)\cong\mathbb P^1$. (See Appendix~\ref{subsec:additionaldivisors} for additional background.)
\begin{definition} \label{def:Cartierdivisorintersectcurve}
    (\textit{Intersection number}) Let $C=V(\tau)$ be the complete torus-invariant curve in $X_\Sigma$ coming from the wall $\tau=\sigma \cap \sigma'$. Let $D$ be a Cartier divisor with Cartier data $m_\sigma, m_{\sigma'} \in M$ corresponding to $\sigma,\sigma' \in \Sigma(n)$. Also pick $u \in \sigma' \cap N$ that maps to the minimal generator of $\bar{\sigma'}:= \sigma'/ \tau \subseteq N(\tau)_\R$. Then 
    $$
    D \cdot C \coloneqq \langle m_\sigma - m_{\sigma'}, u \rangle \in \Z.
    $$
\end{definition}

\begin{remark}\label{rmk:intersectionnumber}
    More intuitively, we may interpret the intersection number as the extent to which the support function $\varphi_D$ bends along the wall $\tau=\sigma \cap \sigma'$.
\end{remark}

\begin{proposition} \label{prop:intersectionnumber}(Linearity of intersection number) Let $C$ be an irreducible complete curve and $D,E$ Cartier divisors on the toric variety $X_\Sigma$ corresponding to a fan $\Sigma$. Then
$$
(D+E) \cdot C=D\cdot C+E\cdot C.
$$    
\end{proposition}

The intersection numbers introduced above will be used to characterize
functions realizable by unbiased shallow ReLU neural networks.
To carry out these computations efficiently, it is convenient to first
normalize the network representation.

Indeed, different shallow networks may realize the same piecewise linear
function while containing redundant neurons, non-primitive weight vectors,
or repeated hyperplanes. Since these redundancies do not affect the
associated ReLU Cartier divisor, we may remove them without changing the
underlying function.

This leads to the notion of a reduced ReLU representation, which provides
a canonical form for the calculations appearing in the proofs of
Theorems~\ref{thm:onehiddenlayer} and
\ref{thm:onehiddenlayerconverse}.

\begin{definition} \label{def:reducedReLU}
    Let $f\in \operatorname{ReLU}_{\vect{0}}^\Q(n_0,1)$and $f_\theta \in \mathcal{F}_{0}^\Q$ realizing $f$:
$$
f_\theta: \R^{n_0} \xrightarrow{\varsigma \circ L_1} \R^{n_1} \xrightarrow{L_2} \R,
$$
and let $\{a_{ij}\}_{i,j=1}^{n_1,n_0},\{b_{1i}\}_{i=1}^{n_1}$ denote the weights of $L_1$ and $L_2$ respectively. Then $f_\theta$ is said to be a \emph{reduced ReLU representation}, denoted as $f_\theta^{\text{red}}$, if
\begin{itemize}
    \item there is no zero row in $L_1$,
    \item $a_{ij}'$s are all integers,
    \item for $i \in\{1,\ldots,n_1\},\gcd(a_{i1},\ldots,a_{in_0})=1$, and
    \item any two rows of $L_1$ are not parallel to each other of the same direction.
\end{itemize}
\end{definition}
\begin{lemma} \label{lem:reducedReLU}(\textit{Algorithm towards reduced ReLU representation)}
    
    Let $f\in \operatorname{ReLU}_{\vect{0}}^\Q(n_0,1)$and $f_\theta \in \mathcal{F}_{0}^\Q$ realizing $f$:
$$
f_\theta: \R^{n_0} \xrightarrow{\varsigma \circ L_1} \R^{n_1} \xrightarrow{L_2} \R,
$$
and let $\{a_{ij}\}_{i,j=1}^{n_1,n_0},\{b_{1i}\}_{i=1}^{n_1}$ denote the weights of $L_1$ and $L_2$ respectively. Run the following algorithm:
\begin{itemize}
    \item If there is a zero row, say $j$th row in $L_1$, then we delete this $j$th row in $L_1$ and the $j$th entry in $\R^{n_1}$, and modify $L_2$ to be a $1 \times (n_1-1)$ matrix by deleting the $j$th entry of $L_2$.
    \item If there exist $k_1,k_2 \in \{1,\ldots,n_1\}$ with $k_1 \not = k_2$ such that the row-wise weights in the $k_1-$th row and the $k_2-$row are parallel of the same direction, i.e., $(a_{k_11},\ldots,a_{k_1n_0})=k(a_{k_21},\ldots,a_{k_2n_0})$ for some $k \in \Q_{>0}$, then delete the $k_1-$th row in $L_1$ and $\R^{n_1}$, delete $b_{1k_1}$ and modify $b_{1k_2}$ to $b_{1k_2}':=b_{1k_2}+k\cdot b_{1k_1}$.
    \item Write $\{a_{ij}\}$ as fractions, $a_{ij}=\frac{c_{ij}}{d_{ij}}$, where $c_{ij},d_{ij} \in \Z$. Let $l=$lcm$[(d_{ij})_{i,j=1}^{n_1,n_0}]$, and for each $i \in \{1,\ldots,n_1\}$, let $l_i=\gcd(l\cdot a_{i1},\ldots,l\cdot a_{in_0})$. Modify each $a_{ij}$ to $a_{ij}':=a_{ij} \cdot \frac{l}{l_i}$, and modify each $b_{1i}$ to $b_{1i}':=b_{1i}\cdot \frac{l_i}{l}$.
    \end{itemize} 
Then the output of the above algorithm is $f_\theta^{\text{red}}$, i.e., the \emph{reduced ReLU representation of $f_\theta$}. 
\end{lemma}

\begin{remark}\label{rmk:reducedReLU}
    The output piecewise linear functions of $f_\theta$ and $f_\theta^{\text{red}}$ are the same, which is $f$.
\end{remark}

\begin{lemma}\label{lem:oppositehyperplane}
    Let
\[
f_\theta^{\mathrm{red}}:\mathbb{R}^{n_0}
\xrightarrow{\varsigma\circ L_1}
\mathbb{R}^{n_1}
\xrightarrow{L_2}
\mathbb{R}
\]
be a reduced representation of an unbiased shallow ReLU neural network. Suppose there exist distinct indices $ k_1,k_2\in\{1,\ldots,n_1\}$
such that
\[
(a_{k_11},\ldots,a_{k_1n_0})
=
-(a_{k_21},\ldots,a_{k_2n_0}).
\]
Let \(f_{k_1}\) and \(f_{k_2}\) denote the piecewise linear functions computed by the \(k_1\)-th and \(k_2\)-th neurons in the hidden layer, respectively, and let
\(D_{f_{k_1}}\) and \(D_{f_{k_2}}\) be the corresponding Cartier divisors on
\(\Sigma_{f_\theta^{\mathrm{red}}}\).
Then
\[
D_{f_{k_1}}=D_{f_{k_2}}.
\]
\end{lemma}

\begin{proof}
    If $(a_{k_11},\ldots,a_{k_1n_0})=-(a_{k_21},\ldots,a_{k_2n_0})$, then $H_{k_1}^{(1)}=H_{k_2}^{(1)}$ and
    \begin{align*}
     f_{k_2}&=0\oplus x_1^{a_{k_21}}\cdots x_{n_0}^{a_{k_2n_0}}=0\oplus x_1^{-a_{k_11}}\cdots x_{n_0}^{-a_{k_1n_0}}\\&=\frac{0\oplus x_1^{a_{k_11}}\cdots x_{n_0}^{a_{k_1n_0}}}{x_1^{a_{k_11}}\cdots x_{n_0}^{a_{k_1n_0}}}=\frac{f_{k_1}}{x_1^{a_{k_11}}\cdots x_{n_0}^{a_{k_1n_0}}}.   
    \end{align*}

    Thus, $f_{k_1}$ and $f_{k_2}$ differ by a linear term, and by \eqref{eq:functiondivisor}, we have $D_{f_{k_1}}=D_{f_{k_2}}$ as desired.
\end{proof}

\begin{lemma}\label{lem:zerointersection}
   Let
\[
f_\theta^{\mathrm{red}}:\mathbb{R}^{n_0}
\xrightarrow{\varsigma\circ L_1}
\mathbb{R}^{n_1}
\xrightarrow{L_2}
\mathbb{R}
\]
be a reduced representation of an unbiased shallow ReLU neural network. Let \(f_j\) denote the piecewise linear function computed by the \(j\)-th hidden neuron, and let \(D_{f_j}\) be the corresponding Cartier divisor on the ReLU fan \(\Sigma_{f_\theta^{\mathrm{red}}}\). Suppose \(\tau\in H_i^{(1)}\), where \(i\neq j\) and \(H_i^{(1)}\neq H_j^{(1)}\). Then
\[
D_{f_j}\cdot V(\tau)=0.
\]
\end{lemma}

\begin{proof}
    Let $\tau=\sigma \cap \sigma'$. Note that for $D_{f_j}$, the difference of the Cartier data $m_\sigma - m_{\sigma'}=0$. Then by definition of intersection number, we see that $D_{f_j}\cdot V(\tau)=0$.
\end{proof}

\OneHiddenLayer*

\begin{proof}
    Let $\{a_{ij}\}_{i,j=1}^{n_1,n_0},\{b_{1i}\}_{i=1}^{n_1}$ denote the weights of $L_1$ and $L_2$ respectively. Denote the piecewise linear functions in the hidden layer by $f_i\coloneqq 0 \oplus x_1^{a_{i1}}\cdots x_{n_0}^{a_{in_0}}$. Then $f_i$ corresponds to the Cartier divisor $D_{f_i}$ supported on the ReLU fan $\Sigma_{f^\text{red}_\theta}$ with Cartier data on one side of the hyperplane $H^{(1)}_i$ equal to $a_{i1}e_1^\vee+\ldots +a_{in_0}e_{n_0}^\vee$ while on the other side equal to $0$. Note that 
    $$
    D_f=\sum_{i=1}^{n_1}b_{1i}D_{f_i}.    
    $$ Then by Lemma \ref{lem:zerointersection} and linearity of intersection number with respect to divisors as in Proposition \ref{prop:intersectionnumber}, we may conclude that for any two walls $\tau_1,\tau_2 \subseteq H^{(1)}_i$, if there exist $i'\not =i$ such that $H_i^{(1)}=H_{i'}^{(1)}$, then we have
    $$
    D_f\cdot V(\tau_1)=b_{1i} [D_{f_i} \cdot V(\tau_1)]+b_{1i'}[D_{f_{i'}}\cdot V(\tau_1)]
    $$
    and 
    $$
    D_f\cdot V(\tau_2)=b_{1i}[D_{f_i} \cdot V(\tau_2)]+b_{1i'}[D_{f_{i'}}\cdot V(\tau_2)].
    $$
    By Lemma \ref{lem:oppositehyperplane}, we have $D_{f_i}=D_{f_{i'}}$ and hence
    $$
    D_f\cdot V(\tau_1)=(b_{1i}+b_{1i'}) [D_{f_i} \cdot V(\tau_1)]
    $$
    and 
    $$
    D_f\cdot V(\tau_2)=(b_{1i}+b_{1i'})[D_{f_i} \cdot V(\tau_2)].
    $$
    Let $\tau_1=\sigma_1 \cap \sigma_1'$. We may assume that the Cartier data of $D_{f_i}$ on $\sigma_1$ and $\sigma_1'$ are $m_{\sigma_1}=a_{i1}e_1^\vee+\ldots +a_{in_0}e_{n_0}^\vee,m_{\sigma_1'}=0$. Then by definition of intersection number,
    $$
    D_{f_i} \cdot V(\tau_1)\coloneqq \langle m_{\sigma_1}-m_{\sigma_1'},u_1\rangle=\langle a_{i1}e_1^\vee+\ldots +a_{in_0}e_{n_0}^\vee,u_1\rangle
    $$
    where $u_1 \in \sigma_1' \cap N$ maps to the minimal generator of $\bar{\sigma_1'}\subseteq N(\tau_1)_\R$. As the set up in Remark \ref{rmk:pairing}, 
    since $m_{\sigma_1'}=0$, the cone $\sigma_1'$ lies on the side of $H_i^{(1)}$ where $\langle a_{i1}e_1^\vee+\ldots+ a_{in_0}e_{n_0}^\vee,x\rangle\leq 0$. Moreover, since the image of $u_1$ is a nonzero generator of $\sigma_1'/\tau_1$, we have
    \[
    \left\langle
    a_{i1}e_1^\vee+\ldots+
    a_{in_0}e_{n_0}^\vee,u_1
    \right\rangle<0.
    \]
    Let
    $
    N_{\tau_1}
    \coloneqq N\cap
    \operatorname{Span}_{\R}(\tau_1)
    $
    be the saturated sublattice associated
    with $\tau_1$. Since $\tau_1$ is a
    codimension-one cone, $N_{\tau_1}$ has
    rank $n_0-1$.

    Let
    $
    \left\{
    w_l=\sum_{j=1}^{n_0}\omega_{lj}e_j
    \right\}_{l=1}^{n_0-1}
    $
    be a $\Z$-basis of $N_{\tau_1}$, and write
    $u_1=\sum_{j=1}^{n_0}c_je_j$.
    Consider the following
    $n_0\times n_0$ integral matrix
    $$
    A=
    \begin{bmatrix}
        c_1 & \cdots & c_{n_0}\\
        \omega_{11} & \cdots & \omega_{1n_0}\\
        \vdots & \cdots & \vdots\\
        \omega_{(n_0-1)1} & \cdots
        & \omega_{(n_0-1)n_0}
    \end{bmatrix}.
    $$

    By the definition of intersection number,
    the image of $u_1$ in
    $N/N_{\tau_1}$ is the primitive generator
    of $\bar{\sigma_1'}$.
    Since $N/N_{\tau_1}\cong\Z$,
    the image of $u_1$ generates the quotient
    lattice. Together with the fact that
    $w_1,\ldots,w_{n_0-1}$ form a $\Z$-basis
    of $N_{\tau_1}$, this implies that
    $u_1,w_1,\ldots,w_{n_0-1}$ form a
    $\Z$-basis of $N$.
    Consequently, the rows of $A$ form a $\Z$-basis of $\Z^{n_0}$, so $A\in\operatorname{GL}_{n_0}(\Z)$. In particular, $A$ admits an integral inverse $B=A^{-1}$ satisfying
    \[
    BA=I_{n_0\times n_0}.
    \]
    Let $v=(a_{i1},\ldots,a_{in_0})^T$ and note that
    $$
    1=\gcd(v^T) \leq \gcd\big ((Av)^T \big) \leq \gcd\big ( (BAv)^T\big )=\gcd(v^T)=1.
    $$
    Thus $\gcd(Av)^T=1$. Let us examine $Av$ as follows:
    $$
    Av=\begin{bmatrix}
        \sum_{k=1}^{n_0} c_ka_{ik} \\
        \sum_{k=1}^{n_0} \omega_{1k}a_{ik}\\
        \vdots \\
        \sum_{k=1}^{n_0} \omega_{(n_0-1)k}a_{ik}
    \end{bmatrix}.
    $$
    Since $N_{\tau_1}\subseteq\operatorname{Span}_{\R}(\tau_1) =H_i^{(1)}$, we have
    $$
    \sum_{k=1}^{n_0}\omega_{lk}a_{ik}=0,
    \qquad \forall l=1,\ldots,n_0-1.
    $$ It follows that $$
    \gcd(Av)^T = \Big \vert \sum_{k=1}^{n_0}c_ka_{ik} \Big \vert=1.
    $$
    Therefore,
    $$
    D_{f_i} \cdot V(\tau_1)=\langle a_{i1}e_1^\vee+\ldots+ a_{in_0}e_{n_0}^\vee,u_1\rangle=-\Big \vert \sum_{k=1}^{n_0}c_ka_{ik} \Big \vert=-1
    $$
    Similarly, we can compute that $D_{f_i} \cdot V(\tau_2)=-1$. It follows that 
    $$
    D_f \cdot V(\tau_1)=-(b_{1i}+b_{1i'})=D_f \cdot V(\tau_2)
    $$
    as desired. If there does not exist $i' \not = i$ such that $H_i^{(1)}=H_{i'}^{(1)}$, then by similar computation we have
    $$
    D_f \cdot V(\tau_1)=b_{1i} [D_{f_i} \cdot V(\tau_1)]=-b_{1i}=b_{1i}[D_{f_i}\cdot V({\tau_2})]=D_f \cdot V(\tau_2)
    $$
    as desired.
\end{proof}
\OneHiddenLayerConverse*
\begin{proof}
    Say that there are $n_1$ full hyperplanes in $\Sigma$, denoted as $H^{(1)}_i$, for $i \in \{1,\ldots,n_1\}$. Let 
    \[
    H_i^{(1)}\coloneqq
    \left\{x\in\R^{n_0}\mid
    \sum_{j=1}^{n_0}a_{ij}x_j=0\right\},
    \]
    where $\gcd(a_{i1},\ldots,a_{in_0})=1$. 
    Say that for any $\tau_1,\tau_2 \in H_i^{(1)}$ and $D_f\cdot V(\tau_1)=D_f \cdot V(\tau_2)=-t_i$, where $t_i \in \Q$, for $i \in \{1,\ldots,n_1\}$. Then $f$ is realizable by the following ReLU neural network, up to a linear shift:
    $$
    l_\theta :\R^{n_0} \xrightarrow{\varsigma \circ L_1} \R^{n_1} \xrightarrow{L_2} \R
    $$
    where 
    $$L_1=
    \begin{bmatrix}
    a_{11}& \cdots & a_{n_01}\\
    a_{21}& \cdots & a_{2n_0}\\
    \vdots & \cdots & \vdots \\
    a_{n_11} & \cdots & a_{n_1n_0}
    \end{bmatrix}
    $$ and 
    $$L_2=
    \begin{bmatrix}
        t_1 & \cdots & t_{n_1}
    \end{bmatrix}.
    $$
    If the output piecewise linear function of $l_\theta$ is precisely $f$, then we are done. Otherwise, denote by $f'$ the output piecewise linear function of $l_\theta$. Then $f-f'=g$ for $g:\R^{n_0}\to \R$ linear. By Theorem \ref{thm:functiondivisor}, we can also conclude that $f \in \operatorname{ReLU}_{\vect{0}}^\Q(n_0,1)$. Thus, $f \in \operatorname{ReLU}_{\vect{0}}^\Q(n_0,1)$ as desired.
\end{proof}
\begin{example}\label{ex:geometrynonexample}
    According to Theorems \ref{thm:onehiddenlayer} and \ref{thm:onehiddenlayerconverse}, we have an equivalent condition for a function to be realizable by shallow unbiased ReLU neural networks with rational weights. We may interpret the equivalent condition as some kind of symmetry, which is, nevertheless, more than merely a symmetric bent hyperplane arrangement. The following piecewise linear function has a symmetric bent hyperplane arrangement while not realizable by shallow ReLU neural networks with rational weights by Theorem \ref{thm:onehiddenlayer}:
    $$
    \begin{tikzpicture}[scale=0.8]
    \draw[black,thick] (-2.5,0) -- (2.5,0);
    \draw[black, thick] (-2,2) -- (2,-2);
    \draw[black, thick] (-2,-2) -- (2,2);
    \draw[black] (-1.5,0.75) node{$y$};
    \draw[black] (-1.5,-0.75) node{$0$};
    \draw[black] (0, 1.5) node{$x+2y$};
    \draw[black] (1.5,0.75) node{$3y$};
    \draw[black] (1.5,-0.75) node{$-4y$};
    \draw[black] (0,-1.5) node{$2x-2y$};
    \end{tikzpicture}.
    $$
\end{example}
    
Next, we will provide a reinterpretation of Theorem \ref{thm:onehiddenlayer} and Theorem \ref{thm:onehiddenlayerconverse}. This reinterpretation will further allow us to decompose any piecewise linear function $f$ into a part that is realizable by shallow networks, denoted as $f_\text{shallow}$, and another part that is not realizable by shallow networks, denoted as $f_\text{non-shallow}$.

Theorem~\ref{thm:onehiddenlayer} suggests that each full hyperplane
contributes an independent shallow ReLU component.
To make this observation geometric, we associate to each hyperplane a toric
morphism to $\mathbb P^1$. We therefore recall the notion of a toric morphism.

\begin{definition}\label{def:toricmorphism}
Let $\Sigma_1$ and $\Sigma_2$ be fans in $N_1$ and $N_2$ respectively.
A \emph{toric morphism}
\[
\phi:X_{\Sigma_1}\to X_{\Sigma_2}
\]
is a morphism induced by a lattice homomorphism
\[
\overline{\phi}:N_1\to N_2
\]
such that for every cone $\sigma\in\Sigma_1$, there exists a cone
$\sigma'\in\Sigma_2$ satisfying
\[
\overline{\phi}_{\mathbb R}(\sigma)\subseteq \sigma'.
\]
\end{definition}

A toric morphism allows divisor classes to be pulled back from one toric variety to another. Since our goal is to characterize shallow ReLU functions in terms of the divisors arising from such pullbacks, it is convenient to work inside the Picard group.

\begin{definition}\label{def:picardgroup}
The \emph{Picard group} of a toric variety $X_\Sigma$ is $\operatorname{Pic}(X_\Sigma)$, the group of Cartier divisors modulo linear equivalence. By the correspondence \eqref{eq:functiondivisor}, divisor classes may equivalently be viewed as piecewise linear functions modulo linear terms.
We also write
\[
\operatorname{Pic}(X_\Sigma)_\Q
=
\operatorname{Pic}(X_\Sigma)\otimes_\Z \Q.
\]
\end{definition}

We now reinterpret Theorems
\ref{thm:onehiddenlayer}
and
\ref{thm:onehiddenlayerconverse}
in toric-geometric terms.

The key idea is that every full hyperplane of the ReLU fan determines a
toric morphism $\psi_i$ to $\mathbb P^1$. The corresponding pullback divisor
$\psi_i^*(\mathcal O_{\mathbb P^1}(1))$
encodes the contribution of that hyperplane to the ReLU Cartier divisor.

Consequently, shallow realizability can be characterized entirely in terms
of divisors pulled back from products of projective lines.

\begin{corollary}\label{cor:shallownetworkreinterpretation}
    Let $f:\R^{n_0} \map \R$ be a piecewise linear function with finitely many pieces and $\Sigma_f$ be the fan given by extending the non-linear loci of $f$ to be full hyperplanes. Let $H_1,\cdots,H_m$ be the full hyperplanes of $\Sigma_f$. As a result of Theorem \ref{thm:onehiddenlayer} and Theorem \ref{thm:onehiddenlayerconverse}, we conclude the following:
    \begin{align*}
        f &\in \operatorname{ReLU}_{\vect{0}}^\Q(n_0,1) \Leftrightarrow \exists \psi: X_{\Sigma_f} \map \prod_{i=1}^m \mathbb{P}^1, \\
        &\psi=\prod_{i=1}^m \psi_i, \psi_i:X_{\Sigma_f}\map \mathbb{P}^1 \text{ is the toric morphism associated with }H_i,\\
        &D_f=\sum_{i=1}^m a_i \psi_i^*(\mathcal{O}_{\mathbb{P}^1}(1)), \text{ where }a_i=D_f \cdot V(\tau), \text{ for any wall }\tau \subseteq H_i.
    \end{align*}
\end{corollary}
\begin{proof}\label{pf:shallownetworkreinterpretation}
    For each full hyperplane $H_i$, denote the half spaces (convex cones) separated by $H_i$ as $\sigma_i$ and $\sigma_i'$. Then consider the fan $\Sigma_{H_i}=\{\overline{\sigma_i},\overline{\sigma_i'},\overline{\{0\}}\}$, where $\overline{\sigma_i}=\sigma_i/H_i, \overline{\sigma_i'}=\sigma_i'/H_i,\overline{\{0\}}=H_i/H_i$. Then $\psi_i:X_{\Sigma_f} \map X_{\Sigma_{H_i}}=\mathbb{P}^1$ is toric. By Theorem \ref{thm:onehiddenlayer} and Theorem \ref{thm:onehiddenlayerconverse}, we know that $f \in \operatorname{ReLU}_{\vect{0}}^\Q(n_0,1) \Leftrightarrow D_f=\sum_{i=1}^m a_i \psi_i^*(\mathcal{O}_{\mathbb{P}^1}(1)), \text{ where }a_i=D_f \cdot V(\tau), \text{ for any wall }\tau \subseteq H_i$.
\end{proof}
\begin{remark}\label{rmk:p1crosscdotsp1}
    By Corollary \ref{cor:shallownetworkreinterpretation}, we can conclude that any $\Q$-Cartier divisor on $\underbrace{\mathbb{P}^1 \times \cdots \times \mathbb{P}^1}_{n\ \text{times}},\forall n\geq 1,$ is realizable by shallow networks. In contrast, no nonzero divisor class in $\operatorname{Pic}(\mathbb P^n)_\Q$, for $n\geq 2$, is realizable by unbiased shallow ReLU networks.
\end{remark}

Corollary~\ref{cor:shallownetworkreinterpretation}
has an important consequence.
It identifies the subspace of the Picard group generated by divisors
realizable by shallow networks.
Consequently, every divisor class admits a decomposition into a shallow
component and a complementary non-shallow component.

This yields a corresponding decomposition of piecewise linear functions
(modulo linear terms), separating the portion that can be realized by a
single hidden layer from the portion that necessarily requires greater
depth. Let $\operatorname{ReLU}(X_{\Sigma_f},1)_\Q$
denote the subspace of $\operatorname{Pic}(X_{\Sigma_f})_\Q$
consisting of divisor classes realizable by
shallow ReLU networks. Then we know $\operatorname{ReLU}(X_{\Sigma_f},1)_\Q$ is the subspace of $\operatorname{Pic}(X_{\Sigma_f})_\Q$ generated by $\psi_i^*(\mathcal{O}_{\mathbb{P}^1}(1)),\forall i=1,\cdots,m$. After choosing a complementary subspace of $\operatorname{ReLU}(X_{\Sigma_f},1)_\Q$, any $D\in \operatorname{Pic}(X_{\Sigma_f})_\Q$ can be written as
$$
D=D_{\text{shallow}}+D_{\text{non-shallow}},
$$
where $D_{\text{shallow}} \in \operatorname{ReLU}(X_\Sigma,1)_\Q$ and $D_{\text{non-shallow}} \in \operatorname{Pic}(X_{\Sigma_f})_\Q \setminus \operatorname{ReLU}(X_{\Sigma_f},1)_\Q$, unless $D_{\text{non-shallow}}=0$ when $D \in \operatorname{ReLU}(X_\Sigma,1)_\Q$. Equivalently, for any piecewise linear function $f$(modulo linear terms), we may decompose $f$ as
$$
f=f_{\text{shallow}}+f_{\text{non-shallow}},
$$
where $f_{\text{shallow}} \in \operatorname{ReLU}_{\vect{0}}^\Q(n_0,1)$ and $f_{\text{non-shallow}} \not \in \operatorname{ReLU}_{\vect{0}}^\Q(n_0,1)$.
\begin{example}\label{ex:decompositionoffunction}
Consider the piecewise linear function as follows
$$
    \begin{tikzpicture}
    \draw[black,thick] (-2.5,0) -- (2.5,0);
    \draw[black, thick] (0,2.5) -- (0,-2.5);
    \draw[black, thick] (-2,-2) -- (2,2);
    \draw[black] (-1.25,1.25) node{$2y$};
    \draw[black] (0.6,1.35) node{$x+2y$};
    \draw[black] (1.5,0.75) node{$3y$};
    \draw[black] (1.25,-1.25) node{$-2y$};
    \draw[black] (-0.6,-1.35) node{$x-2y$};
    \draw[black] (-1.25,-0.55) node{$-y$};
    \draw[black] (2.5, 0.25) node{$H_1$};
    \draw[black] (1.85, 2.2) node{$H_2$};
    \draw[black] (-0.3, 2.5) node{$H_3$};
    \draw[black] (0.25,2.4) node{$\rho_2$};
    \draw[black] (0.25,-2.4) node{$\rho_5$};
    \draw[black] (2.5,-0.25) node{$\rho_0$};
    \draw[black] (-2.5,-0.25) node{$\rho_3$};
    \draw[black] (2.1,1.65) node{$\rho_1$};
    \draw[black] (-1.5,-1.85) node{$\rho_4$};
    \end{tikzpicture}.
$$ The toric variety $X_{\Sigma_f}$ is $\operatorname{Bl}_{\{p_1,p_2\}}(\mathbb{P}^1 \times \mathbb{P}^1)$. We can compute that $\dim \operatorname{Pic}(X_{\Sigma_f})=4$ and $\operatorname{ReLU}(X_{\Sigma_f},1)_\Q$ is a 3-dimensional subspace generated by $\psi_i^*(\mathcal{O}_{\mathbb{P}^1}(1))$ for $\psi_i: X_{\Sigma_f} \map X_{\Sigma_{H_i}}=\mathbb{P}^1$ for $i=1,2,3$. Let $D_i$ denote the $T_N-$invariant prime divisor associated with $\rho_i$ (Definition \ref{def:primedivisor}). Take a basis $\{D_0,D_1,D_2,D_5\}$ for $\operatorname{Pic}(X_{\Sigma_f})$, after choosing a complementary subspace of $\operatorname{ReLU}(X_{\Sigma_f},1)_\Q$, and then we can decompose $D_f=D_{\text{shallow}}+D_{\text{non-shallow}}$, where
$$
D_{\text{shallow}}=-4(D_1+D_2)=-4\psi_1^*(\mathcal{O}_{\mathbb{P}^1}(1)),
$$and 
$$
D_{\text{non-shallow}}=-(D_0+D_1+D_5)+(D_0+D_1+D_2).
$$
Moreover, $D_0+D_1+D_5=p_1^*(\mathcal{O}_{\mathbb{P}^2}(1))$, where
\[
p_1:X_\Sigma \to \mathbb{P}^2
\]
is the toric morphism induced by
\[
\begin{tikzpicture}[scale=0.4,baseline=(current bounding box.center)]
    \draw[black,thick] (-2.5,0) -- (2.5,0);
    \draw[black,thick] (0,2.5) -- (0,-2.5);
    \draw[black,thick] (-2,-2) -- (2,2);
\end{tikzpicture}
\qquad
\vcenter{\hbox{$\longrightarrow$}}
\qquad
\begin{tikzpicture}[scale=0.4,baseline=(current bounding box.center)]
    \draw[black,thick] (0,0) -- (2.5,0);
    \draw[black,thick] (0,0) -- (0,2.5);
    \draw[black,thick] (0,0) -- (-2,-2);
\end{tikzpicture},
\]
and $D_0+D_1+D_2=p_2^*(\mathcal{O}_{\mathbb{P}^2}(1))$, where 
\[
p_2:X_\Sigma \to \mathbb{P}^2
\]
is the toric morphism induced by
\[
\begin{tikzpicture}[scale=0.4,baseline=(current bounding box.center)]
    \draw[black,thick] (-2.5,0) -- (2.5,0);
    \draw[black,thick] (0,2.5) -- (0,-2.5);
    \draw[black,thick] (-2,-2) -- (2,2);
\end{tikzpicture}
\qquad
\vcenter{\hbox{$\longrightarrow$}}
\qquad
\begin{tikzpicture}[scale=0.4,baseline=(current bounding box.center)]
    \draw[black,thick] (0,0) -- (-2.5,0);
    \draw[black,thick] (0,0) -- (0,-2.5);
    \draw[black,thick] (0,0) -- (2,2);
\end{tikzpicture}.
\]
By Remark \ref{rmk:p1crosscdotsp1}, we know that any $\Q-$Cartier divisor on $\mathbb{P}^2$ is not realizable by shallow network, and hence we can conclude that both $D_0+D_1+D_5$ and $D_0+D_1+D_2$ are not realizable by shallow networks. 

Writing in terms of piecewise linear function (modulo linear terms), we decompose $f=f_{\text{shallow}}+f_{\text{non-shallow}}+\text{linear term}$, where
$$
f_{\text{shallow}}=\max\{0,4y\},
$$and
$$
f_{\text{non-shallow}}=\max\{0,x,x-y\}-\max\{0,x,y\},
$$ and
$$
\text{linear term}=-y.
$$
\end{example}
\newpage

\appendix
This appendix collects several auxiliary constructions used throughout the paper. The material is included for completeness and to keep the main exposition focused on the toric-geometric aspects of ReLU neural networks. Readers familiar with these topics may safely skip this appendix on a first reading.
\section{Additional ReLU Neural Network Background}\label{appendix:relubackground}
We first recall two constructions associated with a ReLU neural network: the bent hyperplane arrangement and the canonical polyhedral complex. Both provide combinatorial descriptions of how the network partitions the input space into linearity regions.
\begin{definition}\label{def:benthyperplanearrangement}
    Let 
$$
    f_\theta:\R^{n_0} \xrightarrow{F_1 \coloneqq \varsigma \circ A_1}\R^{n_1} \xrightarrow{F_2 \coloneqq \varsigma \circ A_2} \R^{n_2} \xrightarrow{F_3 \coloneqq \varsigma \circ A_3} \cdots \xrightarrow{F_k \coloneqq \varsigma \circ A_k} \R^{n_k} \xrightarrow{A_{k+1}} \R
$$
be a ReLU neural network and let $\mathcal{A}^{(i)}=\{H^{(i)}_1,\ldots,H^{(i)}_{m_i}\}$ denote the hyperplane arrangement in $\R^{n_{i-1}}$, for any $i\in \{2,\ldots,k\}$. \emph{A bent hyperplane associated with the $i$th layer of $f$}, for any $i\in \{2,\ldots,k\}$, is the preimage in $\R^{n_0}$ of any hyperplane $H^{(i)}_j\subseteq \R^{n_{i-1}}$ associated with the $i$th layer map:
$$
(F_{i-1} \circ \cdots F_1)^{-1}(H_j^{(i)}).
$$
Note that the hyperplanes $H_j^{(1)} \subseteq \R^{n_0}$, for any $j \in \{1,\ldots,m_1\}$ associated with the first layer, do not bend, we however for the sake of consistency call $H_j^{(1)} \in \mathcal{A}^{(1)}$ the \emph{bent hyperplanes associated with the first layer of $f_\theta$}.
\newline

We will denote by $\mathcal{B}_{f_\theta}^{(i)}\coloneqq \bigcup_{j=1}^{m_i} (F_{i-1} \circ \cdots F_1)^{-1}(H_j^{(i)})$ for $i \in \{2,\ldots,k\}$ or $\mathcal{B}_{f_\theta}^{(1)}\coloneqq \bigcup_{j=1}^{m_1} H_j^{(1)}$ the \emph{bent hyperplane arrangement associated with the $i$th or the $1$st layer of $f_\theta$}. 

We will denote by $\mathcal{B}_{f_\theta}\coloneqq \bigcup_{i=1}^{k}\mathcal{B}_{f_\theta}^{(i)}$ the \emph{bent hyperplane arrangement associated with $f_\theta$}.
\end{definition}

\begin{remark}\label{rmk:benthyperplane}
    Since there is no activation function on the last layer map to induce any additional loci of nonlinearity, the last layer will make no contribution to the bent hyperplane arrangement.
\end{remark}

While bent hyperplanes record the loci where nonlinearities occur, the canonical polyhedral complex records the full polyhedral decomposition of the input space induced by the network. We now recall this construction.

\begin{definition}[ \cite{grigsby2022transversality}]\label{appendixdef:canonicalpolyhedralcomplex}
    Let 
$$
    f_\theta:\R^{n_0} \xrightarrow{F_1\coloneqq \varsigma \circ A_1}\R^{n_1} \xrightarrow{F_2\coloneqq \varsigma \circ A_2} \R^{n_2} \xrightarrow{F_3\coloneqq\varsigma \circ A_3} \cdots \xrightarrow{F_k\coloneqq\varsigma \circ A_k} \R^{n_k} \xrightarrow{A_{k+1}} \R
$$
be a ReLU neural network. For $i \in \{1,2,\ldots,k\}$, denote by $R^{(i)}$ the polyhedral complex in $\R^{n_{i-1}}$ induced by the hyperplane arrangement associated to the $i-th$ layer map $F_i$. Inductively define polyhedral complexes, $\mathcal{C}(F_i \circ \cdots \circ F_1)$ in $\R^{n_0}$ as follows: Set
\begin{itemize}
    \item $\mathcal{C}(F_1)\coloneqq R^{(1)}$, and
    \item $\mathcal{C}(F_i \circ \cdots \circ F_1)\coloneqq \mathcal{C}(F_{i-1} \circ \cdots \circ F_1)_{\in R^{(i)}}$ for $i=2,\ldots,k$.
\end{itemize}
The \emph{canonical polyhedral complex} associated to $f_\theta$ is $\mathcal{C}_{f_\theta}\coloneqq \mathcal{C}(F_k \circ \cdots \circ F_1)$.
\end{definition}

\begin{remark}\label{rmk:levelsetcomplex}
    In Definition \ref{appendixdef:canonicalpolyhedralcomplex}, the polyhedral complex $\mathcal{C}(F_{i-1}\circ \cdots \circ F_1)_{\in R^{(i)}}$ is the \emph{level set complex}, which is defined as 
   
\begin{align*}
    &\mathcal{C}(F_{i-1}\circ\cdots\circ F_1)_{\in R^{(i)}}\\
    &\quad\coloneqq
    \left\{
    S\cap(F_{i-1}\circ\cdots\circ F_1)^{-1}(Y)
    \;\Big\vert\right.\\
    &\qquad\left.
    S\in\mathcal{C}(F_{i-1}\circ\cdots\circ F_1),
    \;Y\in R^{(i)}
    \right\}.
\end{align*}

    More details can be found in \cite{grigsby2022transversality}.
\end{remark}
\section{Additional Toric Geometry Background}\label{appendix:toricbackground}
This appendix collects several standard notions from toric geometry that are used throughout the paper but are not essential for a first reading. We begin with characters and co-characters of a torus, which give rise to the lattices $M$ and $N$ that appear throughout toric geometry.
\subsection{Characters and Co-characters}\label{subsec:charactersandcocharacters}
\begin{definition}\label{def:affinevariet}
   Let $S=\C[x_1,x_2,\cdots,x_n]$. Then $S$ is the \emph{coordinate ring} of an affine $n$-space over $\C$, denoted as $\mathbb{A}^n$. An \emph{affine variety} V is defined to be the vanishing locus of finitely many polynomials in S. More specifically, let $B$ be a set of finitely many polynomials in $S$, then 
    $$
    V=\mathbb{V}(B) \coloneqq \{ p \in \mathbb{A}^n \big \vert f(p)=0 \text{ for any f }\in B\}.
    $$
\end{definition}
\begin{definition}\label{def:torus}
    The affine variety $(\C^*)^n$ is a group under the component-wise multiplication. A \emph{torus} T is an affine variety which is isomorphic to $(\C^*)^n$, where T inherits a group structure from the isomorphism, i.e., T acts on itself by multiplication. Associated to a torus T are its characters and co-characters defined as follows.
\end{definition}
The character and co-character lattices provide the combinatorial data underlying toric geometry. They will serve as the ambient lattices in which cones, fans, and support functions are defined.
\begin{definition}\label{def:character}
    A \emph{character} of a torus T is a morphism $\chi^m: T\iso (\C^*)^n \map \C^*$ that is a group homomorphism. 
\end{definition}   
\begin{remark}\label{rmk:character}
    For instance, $m=(a_1,a_2,\ldots,a_n) \in \Z^n$ gives a character $\chi^m: (\C^*)^n \map \C^*$ defined by
    $$
    \chi^m(t_1,t_2,\ldots,t_n)=t_1^{a_1}t_2^{a_2}\cdots t_n^{a_n},
    $$
    where $(t_1,t_2,\ldots,t_n) \in T\iso (\C^*)^n$. Indeed, all characters of $(\C^*)^n$ arise this way. Thus the characters of $(\C^*)^n$ form a group isomorphic to $\Z^n$. Furthermore, the group of characters of a torus T is denoted as M and $M \iso \Z^n$.
\end{remark}
\begin{definition}\label{def:cocharacter}
    A \emph{co-character} of a torus T is a morphism $\lambda^u: \C^* \map T\iso (\C^*)^n$ that is a group homomorphism. 
\end{definition}
\begin{remark}\label{rmk:cocharacter}
For instance, $u=(b_1,b_2,\ldots,b_n) \in \Z^n$ gives a co-character $\lambda^u: \C^* \map (\C^*)^n$ defined by
    $$
    \lambda^u(t)=(t^{b_1},t^{b_2},\ldots,t^{b_n}),
    $$
    where $t \in \C^*$. Indeed, all co-characters of a torus $T \iso (\C^*)^n$ arise this way. Thus the co-characters of $(\C^*)^n$ form a group isomorphic to $\Z^n$ and the group of co-characters of a torus T is denoted as $N$ where $N \iso \Z^n$.
\end{remark}

Characters and co-characters are naturally dual to one another. This duality is encoded by a canonical bilinear pairing between the character lattice $M$ and the co-character lattice $N$.

\begin{definition}\label{def:characterpairing}
    There is a natural \emph{bilinear pairing} $\langle,\rangle: M \times N \to \Z$ defined as follows. Given a character $\chi^m$ and a co-character $\lambda^u$, the composition 
    \begin{align*}
        \chi^m \circ \lambda^u &:\C^* \to \C^*\\
        & t \mapsto t^l
    \end{align*}
    is a character of the torus $T \iso \C^*$, for some $l \in\Z$. Then $\langle m,u \rangle =l$.
\end{definition}

\subsection{Additional Facts on Cones}\label{subsec:additionalfactsoncones}
The main text introduces the basic notions of cones used in the construction of toric varieties. We record here several additional constructions and properties that will be used in the divisor-theoretic discussion later in the appendix.
\begin{definition}\label{def:perpendicularcone}
    Let $\sigma\subseteq N_\R$ be a polyhedral cone. We define
    $$
    \sigma^\perp \coloneqq \{m \in M_\R \mid \langle m,u \rangle=0 \text{ for all }u \in \sigma \}.
    $$
\end{definition}

The subspace $\sigma^\perp$ records the characters that are constant along the directions of $\sigma$. It will later be used in the construction of torus orbits and torus-invariant divisors.
\begin{definition}\label{def:dimensionofcone}
    The \emph{dimension} dim $\sigma$ of a polyhedral cone $\sigma$ is the dimension of the smallest subspace $W=$ Span($\sigma$) of $N_\R$ containing $\sigma$. We call Span$(\sigma)$ the span of $\sigma$, which stands for $\R-$linear span.
\end{definition}
\begin{remark}\label{rmk:stronglyconvexitydual}
If a strongly convex cone $\sigma$ has maximal dimension, then so does $\sigma^\vee$.
\end{remark}

\subsection{Affine Semigroups and Semigroup Algebras}\label{subsec:monoid}

This subsection records several standard notions from the theory of affine semigroups and semigroup algebras. These constructions underlie the correspondence
\[
\sigma \longmapsto S_\sigma \longmapsto \C[S_\sigma] \longmapsto U_\sigma
\]
used throughout toric geometry to associate algebraic-geometric objects to rational polyhedral cones. The main text contains a representative example illustrating this construction. Here we collect the underlying algebraic notions together with several additional examples for readers who would like further background.
\begin{definition} \label{def:monoid}
A \emph{monoid} $S$ is a set equipped with a commutative and associative binary operation, denoted as $+$, and an identity element, denoted as $0$. A \emph{submonoid} $S'$ of a monoid $A$ is a subset $S'\subseteq S$ that is closed under the monoid operation $+$ and contains the identity element $0$ of $S$.
\end{definition}

\begin{definition} \label{def:monoindmorphism}
A \emph{monoid morphism} between two monoids $(S,+_S)$ and $(S',+_{S'})$ is a function $f:S \map S'$ such that
\begin{itemize}
    \item $f(x +_S y)= f(x)+_{S'} f(y),\forall x, y \in S$,
    \item $f(0_S)=0_{S'}$
\end{itemize}
We denote the set of all monoid morphism from $S$ to $S'$ as \emph{Hom$_{\text{Mon}}(S,S')$}.
\end{definition}
\begin{remark}\label{rmk:monoidmorphism}
    When $S,S'$ are groups, which are monoids with elements having inverses, we have $\text{Hom}_{\text{Mon}}(S,S')=\text{Hom}_{\text{Grp}}(S,S')$. Hence, for the maximal torus $T_N$, we have
    $$
    T_N \coloneqq \text{Hom}_{\text{Grp}}(M,\C^*) = \text{Hom}_{\text{Mon}}(M,\C) \iso (\C^*)^n
    $$
\end{remark}
\begin{definition} \label{def:generatorofmonoid}
    Let $S$ be a monoid. A subset $\mathcal{A} \subseteq S$ is said to \emph{generate $S$} if 
        $$
        \N\mathcal{A} \coloneqq \left \{\sum_{m \in \mathcal{A}} a_m m \big \vert a_m \in \N\} \subseteq S\right \} =S.
        $$
        Moreover, if $\mathcal{A} \subseteq S$ is a finite set such that $\N \mathcal{A}=S$, then we say the monoid $S$ is \emph{finitely generated}.
\end{definition}

\begin{definition}\label{def:affinesemigroup}
    An \emph{affine semigroup} $S$ is a monoid such that
    \begin{itemize}
        \item the monoid is finitely generated, and
        \item the monoid can be embedded in a lattice M as a submonoid.
    \end{itemize}
\end{definition}

When an affine semigroup is embedded in the character lattice $M$, its semigroup algebra can be described explicitly in terms of characters of the torus.

\begin{remark}\label{rmk:semigroupalgebra}
        Let $S\subseteq M$ be an affine semigroup, then
    $$
    \C[S]\coloneqq \left \{\sum_{m \in S}c_m \chi^m \big \vert c_m \in \C \text{ and }c_m=0 \text{ for all but finitely many }m \right\},
    $$
    with multiplication induced by
    $$
    \chi^m \cdot \chi^{m'}=\chi^{m+m'}.
    $$
    If $S=\N \mathcal{A}$ for $\mathcal{A}=\{m_1,\ldots,m_s\}$, then $\C[S]=\C[\chi^{m_1},\ldots,\chi^{m_s}]$.
    \end{remark}

The following additional examples further illustrate how semigroup algebras arise naturally from lattices and cones in toric geometry.

\begin{example}\label{ex:affinealgebra}
     Recall that the character lattice $M \iso \Z^n$, then $M=\N\{\pm e_1^\vee,\ldots,\pm e_n^\vee\}$ as an affine semigroup. Setting $x_i=\chi^{e_i^\vee}$ yields the Laurent polynomial ring
    $$
    \C[M]=\C[x_1^{\pm},\ldots,x_n^{\pm}],
    $$
    which is the coordinate ring of the maximal torus $T_N \iso (\C^*)^n$. The affine semigroup $\N^n\subseteq \Z^n$ gives the polynomial ring
    $$
    \C[\N^n]=\C[x_1,\ldots,x_n].
    $$
\end{example}

\begin{remark}\label{rmk:algebrarelations}
    The notation $\C[S]=\C[\chi^{m_1},\ldots,\chi^{m_s}]$ suppresses any relations among the generators. As Example~\ref{ex:affinealgebrarelation} demonstrates, such relations can occur and are reflected in the defining equations of the corresponding affine toric variety.
\end{remark}

\subsection{Affine Toric Variety Associated to a Cone}\label{subsec:affinecone}

After associating a semigroup algebra to a rational polyhedral cone, one may recover a geometric object from this algebra. This construction produces the affine toric variety associated with the cone. In the traditional algebro-geometric approach, the semigroup algebra serves as the coordinate ring of the resulting affine toric variety.

A key feature of affine toric varieties is that they contain an algebraic torus as a dense open subset. The following criterion relates the dimension of the torus to the geometry of the underlying cone.

\begin{proposition}\label{prop:affinetoricvariety}
    Let $U_\sigma$ be an affine toric variety, then we have:
    $$
    \text{dim }U_\sigma=n \Longleftrightarrow \text{the torus of }U_\sigma \text{ is }T_N \Longleftrightarrow \sigma \text{ is strongly convex}.
    $$
\end{proposition}
\begin{remark}\label{rmk:affinetoricvariety}
    In the standard algebraic-geometric formulation of toric geometry, the affine toric variety associated with a rational polyhedral cone $\sigma$ is defined as
$$
U_\sigma=\operatorname{Spec}(\C[S_\sigma])
=\operatorname{Spec}(\C[\sigma^\vee\cap M]).
$$
The description in Definition \ref{def:affinetoricvariety} is equivalent to this construction and emphasizes the point-set perspective.
\end{remark}

The following example illustrates how familiar algebraic varieties arise naturally from simple cones.
\begin{example}\label{ex:cx+-}
    Let $\sigma=\{0\}\subseteq N_\R \iso \R^n$. Then $S_\sigma=\sigma^\vee \cap M=M$. As seen in Example \ref{ex:affinealgebra}, we have the affine toric variety $U_\sigma=T_N\iso (\C^*)^n$ with coordinate ring $\C[M]=\C[x_1^{\pm},\ldots,x_n^{\pm}]$.
\end{example}

\subsection{Toric varieties associated with fans}\label{subsec:toricvarietyfan}
Having associated an affine toric variety to a single cone, we now explain how a fan gives rise to a global toric variety by gluing these affine pieces together. We need the following Proposition, which is essential for gluing the affine varieties $\{U_\sigma\}_{\sigma \in \Sigma(n)}$ to get $X_\Sigma$. Here $\{U_\sigma\}_{\sigma \in \Sigma(n)}$ is called an \emph{affine open cover} of $X_\Sigma$.
\begin{proposition}(Separation Lemma) \label{prop:separationlemma}
    Let $\Sigma$ be a fan. If $\sigma_1,\sigma_2 \in \Sigma$ and $\tau=\sigma_1 \cap \sigma_2$, then 
    $$
    S_{\tau}=S_{\sigma_1}+S_{\sigma_2}\coloneqq \left \{m_1+m_2 \mid m_1 \in S_{\sigma_1}, m_2 \in S_{\sigma_2} \right \}\subseteq M.
    $$
\end{proposition}

To illustrate the gluing construction, let $\sigma_1,\sigma_2\in\Sigma(n)$ and set
$\tau=\sigma_1\cap\sigma_2$. Note that
$
U_{\sigma_1}=\operatorname{Hom}_{\mathrm{Mon}}(S_{\sigma_1},\C),
$
$
U_{\sigma_2}=\operatorname{Hom}_{\mathrm{Mon}}(S_{\sigma_2},\C),
$
and
$
U_{\tau}=\operatorname{Hom}_{\mathrm{Mon}}(S_{\tau},\C).
$
By Proposition~\ref{prop:separationlemma}, we have
$S_\tau=S_{\sigma_1}+S_{\sigma_2}$, and hence
$U_\tau=U_{\sigma_1}\cap U_{\sigma_2}$.
Therefore, $U_{\sigma_1}$ and $U_{\sigma_2}$ may be glued along their common open subset $U_\tau$.
Repeating this construction for all cones of $\Sigma$ yields the toric variety $X_\Sigma$.

\subsection{Additional Background on Divisors}\label{subsec:additionaldivisors}

This subsection recalls the divisor-theoretic background underlying the combinatorial description of Cartier divisors used in the main text. In toric geometry, divisors are encoded by the rays of a fan, and Cartier divisors admit a combinatorial description in terms of lattice data on the maximal cones. We briefly review this correspondence here.

Let $X_\Sigma$ be the toric variety associated with a fan $\Sigma$ and let
$T_N\subseteq X_\Sigma$ be its maximal torus.
The complement
\[
X_\Sigma\setminus T_N
\]
decomposes into irreducible codimension-one components indexed by the rays
$\rho\in\Sigma(1)$.
These components are the $T_N$-invariant prime divisors of $X_\Sigma$.

\begin{definition}\label{def:primedivisor}
    Consider $\rho \in \Sigma(1)$, and suppose $\rho \subseteq \sigma$ for some $\sigma \in \Sigma$. We obtain a \emph{$T_N-$invariant prime divisor} $D_\rho$, defined as the closure of 
    $$
    O(\rho) \coloneqq 
\left\{\, \gamma \in \operatorname{Hom}_{\mathrm{Mon}}(S_\sigma,\mathbb{C}) \;\middle|\;
\gamma(m) \not = 0 \Leftrightarrow m \in \rho^\perp \cap M
\right\},
    $$ i.e., $D_\rho \coloneqq \overline{O(\rho)}$.
\end{definition}

\begin{definition} \label{def:weildivisor}
    A \emph{$T_N-$invariant Weil divisor} $D$ is a finite sum
    $$
    D=\sum_{\rho \in \Sigma(1)} a_\rho D_\rho
    $$
    where $a_\rho \in \Z$ and $D_\rho$'s are $T_N$-invariant prime divisors.
\end{definition}

\begin{example}\label{ex:weildivisor}
    The fan $\Sigma$ in Example \ref{ex:fanofp1} has two rays $\rho_1=$ Cone$(e_1)$ and $\rho_2=$ Cone$(-e_1)$. In this case, $X_{\Sigma}\setminus T_N = \mathbb{P}^1 \setminus \mathbb{C}^*$ consists of two points, which are $D_{\rho_1}=\{\chi^{e_1^\vee}=x=0\}$ and $D_{\rho_2}=\{\chi^{-e_1^\vee}=x^{-1}=0\}$.
\end{example}

\begin{remark}\label{rmk:weildivisor}
    Not every Weil divisor is locally defined by a single equation. Those which are are called Cartier divisors. For toric varieties, Cartier divisors admit a particularly simple combinatorial description in terms of lattice elements on the maximal cones of the fan.
\end{remark}

\begin{definition}\label{appendixdef:cartierdivisor}
    A Weil divisor $D$ is \emph{Cartier} if for an affine open cover $\{U_\sigma\}_{\sigma \in \Sigma(n)}$ of $X_\Sigma$, $D$ is of the form $\{(U_\sigma,\chi^{-m_\sigma})\}$ for some $m_\sigma \in M$ such that $\langle m_\sigma, u_\rho \rangle=-a_\rho$ for all $\rho \in \sigma(1)$, where $u_\rho$ is the minimal generator of the ray $\rho$. And $\{m_\sigma\}$ is called the \emph{Cartier data}. 
\end{definition}

\begin{example}\label{ex:cartierexample}
    Continuing with Example \ref{ex:weildivisor}, the two Weil divisors $D_{\rho_1}$ and $D_{\rho_2}$ are both Cartier. The Cartier data for $D_{\rho_1}$ is $\{m_{\sigma_1}=-e_1^\vee,m_{\sigma_2}=0\}$ and the Cartier data for $D_{\rho_2}$ is $\{m_{\sigma_1}=0,m_{\sigma_2}=e_1^\vee\}$.
\end{example}

Cartier divisors are particularly important because they naturally determine line bundles and associated polyhedra, which form the bridge between toric geometry and the ReLU constructions studied in the main text.

\begin{definition} \label{def:linebundle}
    Let $D$ be a Cartier divisor represented by $\{U_\sigma,m_\sigma\}$ on the toric variety $X_\Sigma$. We associate a \emph{line bundle} $\mathcal{O}_{X_\Sigma}(D)$ to $D$ which is defined to be a functor mapping open affines $U_\sigma$ to $\mathcal{O}_{X_\Sigma}(D)(U_\sigma)$ where
    $$
    \mathcal{O}_{X_\Sigma}(D)(U_\sigma)\coloneqq \C[S_\sigma] \cdot \chi^{m_\sigma}.
    $$
\end{definition}

\begin{remark}\label{appendixrmk:polytopeofdivisor}
Given a Cartier divisor
\[
D=\sum_{\rho\in\Sigma(1)} a_\rho D_\rho,
\]
then the polyhedron $P_D$ introduced in Definition~\ref{def:polytopeofdivisor} can equivalently be described by
\[
P_D
=
\{m\in M_\R \mid
\langle m,u_\rho\rangle \ge -a_\rho
\text{ for all }\rho\in\Sigma(1)\}.
\]
\end{remark}

To study divisors geometrically, one often examines how they interact with curves. In toric geometry, the most important curves are the torus-invariant curves associated with walls of the fan.

\begin{definition} \label{def:torus invariant curve}
    Let $\Sigma$ be a complete fan in $N_\R\iso \R^n$. A curve $C$ is \emph{torus-invariant and complete} if $C=V(\tau) \iso \mathbb{P}^1$ where $\tau=\sigma \cap \sigma' \in \Sigma(n-1),$ and $\sigma,\sigma'\in \Sigma(n)$.
\end{definition}

\begin{remark}\label{rmk:torusinvariantcurve}
    When $\tau=\sigma \cap \sigma'\in \Sigma(n-1)$ and $\sigma,\sigma' \in \Sigma(n)$, $\tau$ is a wall. To see why the toric variety $V(\tau) \iso \mathbb{P}^1$, the co-character lattice of $V(\tau)$ is $N(\tau) \coloneqq N/N_{\tau}$, where $N_{\tau}=$ Span$(\tau) \cap N$. Since $\tau$ is a wall, we have the torus $N(\tau) \iso \Z$. Furthermore, $\bar{\sigma},\bar{\sigma'}$ are rays that correspond to the rays in the usual fan for $V(\tau) \iso \mathbb{P}^1$, where $\bar{\sigma}\coloneqq \sigma/\tau, \bar{\sigma'}\coloneqq \sigma'/ \tau$.
\end{remark}

The main text primarily works with $\Q$-Cartier divisors. For completeness, we recall how the intersection pairing extends from Cartier divisors to $\Q$-Cartier divisors.

\begin{remark}\label{rmk:QCartierintersection}
    If $D$ is a $\Q-$Cartier divisor, i.e., $lD$ is Cartier for some integer $l>0$. Given a complete torus-invariant curve C, then 
    $$
    D \cdot C \coloneqq \frac{1}{l}(lD) \cdot C \in \Q.
    $$
\end{remark}

\bibliographystyle{plain}
\bibliography{references}
\bigskip

\textsc{Boston College; Department of Mathematics; 521 Maloney Hall; Chestnut Hill, MA 02467}

\textit{Email address:} \texttt{fugh@bc.edu}
\end{document}